\documentclass{amsart}
\usepackage{graphicx}
\usepackage{bbm}
\usepackage{epsfig}
\usepackage{amssymb}
\usepackage{amsmath,enumerate,color}
\usepackage{subfigure}
\usepackage{epstopdf}
\usepackage{amsfonts}
\usepackage{siunitx}
\usepackage{booktabs}
\usepackage{geometry}
\usepackage{comment,enumerate,multicol,xspace}
\DeclareMathOperator*{\esssup}{ess\,sup}
  \newcounter{mnote}
  \let\oldmarginpar\marginpar
    \renewcommand\marginpar[1]{\-\oldmarginpar[\raggedleft\footnotesize #1]%
    {\raggedright\footnotesize #1}}

\newtheorem{theorem}{Theorem}[section]
\newtheorem{lemma}[theorem]{Lemma}
\newtheorem{proposition}[theorem]{Proposition}
\newtheorem{corollary}[theorem]{Corollary}

\theoremstyle{definition}
\newtheorem{definition}[theorem]{Definition}

\theoremstyle{remark}
\newtheorem{remark}[theorem]{Remark}

\numberwithin{equation}{section}

\begin{document}

\title[a posteriori error estimates for fully discrete TFPDEs]{Optimal $L^\infty(L^2)$ and $L^1(L^2)$  a posteriori error estimates for the fully discrete approximations of time fractional parabolic  differential equations}

\author{Jiliang Cao}
\address{Department of Mathematics, Shanghai Normal University, Shanghai, 200234, China}
\email{jiliangcao@shnu.edu.cn}

\author{Wansheng Wang}
\address{Department of Mathematics, Shanghai Normal University, Shanghai, 200234, China}
\email{w.s.wang@163.com}

\author{Aiguo Xiao}
\address{Hunan Key Laboratory for Computation and Simulation in Science and Engineering \& Key Laboratory of
Intelligent Computing and Information Processing of Ministry of Education, Xiangtan University, Xiangtan,
Hunan 411105, China}
\email{xag@xtu.edu.cn}

\thanks{The second author was supported by the
Natural Science Foundation of China (Grant Nos. 12271367, 11771060), Shanghai Science and Technology Planning Projects (Grant No. 20JC1414200), and Natural Science Foundation of Shanghai (Grant No. 20ZR1441200).}
\thanks{The third author was supported by the
 National Natural Science Foundation of China (Grant No. 12071403).}
\thanks{The second author is the corresponding author.}


\date{April 2, 2023.}


\begin{abstract}
We derive optimal order a posteriori error estimates  in the $L^\infty(L^2)$ and $L^1(L^2)$-norms for the fully discrete approximations of time fractional parabolic  differential equations. For the discretization in time, we use the $L1$ methods, while for the spatial discretization, we use standard conforming finite element methods. The linear and quadratic space-time reconstructions are introduced, which are  generalizations of the elliptic space reconstruction. Then the related a posteriori error estimates for the linear  and quadratic space-time reconstructions play key roles in deriving  global and pointwise  final error estimates. Numerical experiments verify and complement our theoretical results.
\end{abstract}

\maketitle

\section{Introduction}\label{sec:1}

In recent years, fractional calculus has been greatly developed in many areas of science and engineering, from porous media to viscoelastic mechanics, and from non-Newtonian fluid mechanics to control systems (see, e.g., \cite{Alina16,Armour16,Igor15,Du20,Kilbas06,Paris14}). Time fractional parabolic  differential equations (TFPDEs), as an important class of anomalous diffusion equations that describe a very important physical and mechanical process of material migration and transport, have attracted a lot of attention from researchers (see, e.g., \cite{Almeida15,Malinowska15,Tarasov11,Ortigueira11,Ra80,Ralf14}). As pointed out by Crank and Cryer \cite{Cryer06}, the most powerful and effective tool for modeling anomalous diffusion processes is the continuous time random walk model, which can describe different anomalous diffusion processes based on different assumptions about jump lengths and waiting times. When the mean square displacement of a particle is a nonlinear function of time, it is possible to describe more accurately certain physical phenomena whose important feature is that the flux of a particle is related not only to its neighborhood, but also to other points throughout space and its history of change, showing a strong historical memory and long-range correlation.

Since the class of time-dependent problems (\ref{eq1.1}) arise in various models of physical processes (see \cite{Jin19,Li10,Metzler00} and therein), these problems and their numerical approximations have attracted much attention in recent years; see, e.g., recent literature  \cite{Banjai19,Chen19,Cheng20,Jin20,Wang20,Zhang20}. There are several predominant classes of numerical methods for discretizing the time fractional derivative, and we focus here on the convolution quadrature (CQ)  method \cite{Cao16,Chen17,Diethelm06,Jin16,Jin17,Jin18,Lubich86,Lubich88,Lubich04,Zayernouri16,Zeng13,Zeng17}, the spectral methods \cite{Jin19,Li10,Jin20,Diethelm20,Li09,Lischke17,Stynes17}, and the finite difference type methods \cite{Li17,Lin07,Lv16,Sun06,Yan14,Diethelm97,Gao14}, etc.

As we all know, CQ inherits excellent numerical stability property of the underlying schemes for ODEs and spectral methods could achieve high order of accuracy with smooth solutions. However, with the study on the regularity of fractional differential equations, it is found that the fractional derivative have initial singularity, i.e., the solutions have weak singularity at the initial moment, which means that the high regularity requirement of the numerical algorithm is not feasible or meaningful. It is worth mentioning that the most time  stepping schemes (including CQ method and the finite difference type methods) on uniform mesh only exhibits a first-order of accuracy when solving fractional evolution equations  even for smooth source function. Recently, to recover the optimal convergence order for nonsmooth solutions of the equations, some special nonuniform meshes have been proposed by taking into account the initial singularity in the problem (\ref{eq1.1}), such as graded mesh, quasi-graded mesh, general nonuniform mesh, and so on (see, e.g., \cite{Stynes17,Liao18,Kopteva19,Kopteva20,Yan18}). Therefore, it is natural to choose numerical schemes which require low regularity and can be easily extended to nonuniform mesh for solving time fractional differential equations. Unlike spectral methods which require high regularity and CQ method often restricted to uniform mesh, finite difference methods based on piecewise polynomial approximation require lower regularity and can be easily extended to nonuniform mesh. As the most famous finite difference method for solving fractional differential equations, L1 method has been widely concerned and studied because of its flexible construction and implementation (see, e.g., \cite{Li18,Li19,Brunner10,Ford11,Ford12,Jin162}). In this paper, we will focus on this scheme.

A posteriori error estimates can be viewed as such type of error estimates which can be quantified for a given simulation, knowing only the problem data and approximate solution. Such computable a posteriori error estimates have been investigated by many researchers for various numerical methods for integer-order parabolic problems during the last decades (see, e.g., \cite{Akrivis06,Bansch12,Lozinski09,Verfurth03,Wang18,Wang201,Wang21,Wang22}). In many cases, a posteriori error estimates and adaptivity have become very successful tools for efficient numerical computations of linear and nonlinear integer-order problems. For TFPDEs (\ref{eq1.1}), however, to the best of our knowledge, there are very limited papers concerning a posteriori error estimate of numerical methods. For examples, Kopteva \cite{Kopteva21} gave the pointwise-in-time a posteriori error bounds of the L1 method in the $L^2$ and $L^{\infty}$ norms, and designed the adaptive algorithms based on the a posteriori error estimates; Banjai and Makridakis \cite{Banjai22} derived a posteriori error estimates of the L1 method or CQ method  in the $L^2$ and $L^{\infty}$ norms; Franz and Kopteva considered \cite{Franz22}  pointwise-in-time a posteriori error bounds of higher-order discretizations of time-fractional parabolic equations; Kopteva and Stynes \cite{Kopteva22} obtained a posteriori error bounds in $L^2$ and $L^{\infty}$ for the variable-coefficient multiterm time-fractional subdiffusion equations. We introduce reconstruction-based a posteriori error estimates of the L1 method for TFPDEs in \cite{CJL22}. All of the above literatures consider only a posteriori error estimates in the semi-discrete scheme of the subdiffusion equation, which motivates us to further investigate a posteriori error estimates in the fully discrete scheme. In this paper, we introduce the linear and quadratic space-time reconstructions to derive a posteriori error estimates in the $L^\infty(L^2)$ and $L^1(L^2)$-norms for TFPDEs. The main difficulty in deriving the optimal order a posteriori error estimates for the finite difference type methods is to obtain the error equations which involves long-range history dependence. In this paper, we will address this issue. The numerical experiments indicate that the a posteriori estimates for the fully discrete scheme  could achieve optimal convergence order on the graded mesh regardless of whether the solution of TFPDES are smooth or not. Finally, we briefly highlight the main contribution of this paper as follows.

\begin{itemize}
\item{
It is the first time that a posteriori error estimates of the fully discrete scheme are derived for TFPDEs. In contrast to the case of semi discrete scheme, the difficulties in deriving the a posteriori error estimate of the fully discrete scheme mainly come from the elliptic space reconstruction operator $\mathcal R^n$. Firstly, because the domain $V_0^n$ of the elliptic reconstruction operator $\mathcal R^n$ is different at different time nodes $t_n$, the derivative and fractional derivative of the reconstruction solution $U_{\mathcal R}$ and $\widehat U_{\mathcal R}$ containing the elliptic reconstruction operator are more complicated. Secondly, since the elliptic reconstruction operator $\mathcal R^n$ is a continuous abstract operator, it leads to the fact that a posteriori error estimate cannot be computed directly. By introducing the duality technique, we obtain error estimates of the elliptic space operator and its derivative, which is the basis and tool for us to further obtain computable expressions of a posteriori error estimators. More details can be found in the Appendices.
}
\item{
 For TFPDEs, no other work has studied a posteriori error estimate in the $L^1(L^2)$-norm. From Corollarys \ref{cor1} and \ref{cor2}, we  can see that a posteriori error estimate of the fully discrete scheme in the $L^1(L^2)$-norm contains the results of a posteriori error estimate in the $L^2(L^2)$-norm. In this sense, a posteriori error estimate in the $L^1(L^2)$-norm is more applicable than a posteriori error estimate in the $L^2(L^2)$-norm, and  this is also the reason why we derive a posteriori error estimate in the $L^1(L^2)$-norm.
}
\item{
As a conclusion of the above theoretical results, global and pointwise final error estimates of the fully discrete scheme for TFPDEs are obtained.
}
\end{itemize}

The paper is organized as follows. In Section 2, we introduce the notations and preliminaries, and then give the finite element semi-discrete approximations to TFPDEs. In Section 3, we  approximate the Caputo fractional derivative and introduce the linear time reconstruction, quadratic time reconstruction and elliptic space reconstruction . In Section 4, we obtain a posteriori error estimate for the linear space-time reconstruction and quadratic space-time reconstruction in the $L^\infty(L^2)$ and $L^1(L^2)$-norms. In Section 5, we obtain global and pointwise final error estimates for the linear space-time reconstruction and quadratic space-time reconstruction. In Section 6, the numerical examples are presented to demonstrate effectiveness of numerical schemes and confirm the theoretical analysis. Finally, we will give the computable expressions of a posteriori error estimators in the Appendixes.

\section{Finite element semi-discrete approximations to TFPDEs}\label{2}

\subsection{Model problem} The model problem we consider is to find $u:\Omega\times [0,T]\to \mathbb R$ satisfying
\begin{eqnarray}
\partial^\alpha_tu(x,t)+{\mathcal A}u(x,t)&=&f(x,t),\quad (x,t)\in \Omega\times (0,T],\label{eq1.1}\\
u(x,t)&=&0,~~~\qquad (x,t)\in \partial\Omega\times (0,T],\\
u(x,0)&=&u_0(x),\quad x\in \Omega,\label{eq1.3}
\end{eqnarray}
where $\Omega\subset {\mathbb R}^d$, $d\in \mathbb N$, is a bounded convex polygonal domain with the boundary $\partial \Omega$,$~T<\infty$  is a constant, $u_0$ is a given initial function, $f(x,t)$ is a given forcing term. Here operator $\mathcal A$  is of the form,
$$\mathcal Au=-\nabla\cdot(A(x)\nabla u),$$
with $\mathcal A$ being symmetric positive definite, where \textquotedblleft $\nabla$\textquotedblright~ denotes the spatial gradient operator, the matrices $A$ are assumed to be smooth. Concerning the existence and uniqueness results of (\ref{eq1.1})-(\ref{eq1.3}), we refer to \cite{Stynes17} and the references therein.

\subsection{Notations and preliminaries}
Given a Lebesgue measurable set $\omega\subset {\mathbb R}^d$, we denote the Lebesgue space by $L^p(\omega)~(p\ge 1)$ and the standard Sobolev space by $W^{s,p}(\omega)~(s\ge 0,~p\ge 1)$, where their corresponding norms are $\|\cdot\|_{L^p(\omega)}$ and $\|\cdot\|_{s,p,\omega}$. For convenience, we denote the norms of the
spaces $H^s(\omega)=W^{s,2}(\omega)$, $L^\infty(\omega)$ and $L^2(\omega)$ by $\|\cdot\|_{s,\omega}$, $\|\cdot\|_{\infty,\omega}$ and $\|\cdot\|_{\omega}$, respectively. The function space $H^1_0(\omega)$ denotes the space $H^1(\omega)$ vanishing on the boundary of $\omega$, and the inner product of space $L^2(\omega)$ is denoted by $\langle\cdot,\cdot\rangle_\omega$. When $\omega=\Omega$, the letter $\omega$ will be omitted in the subscripts of $\|\cdot\|_{s,p,\omega}$, $\|\cdot\|_{s,\omega}$, $\|\cdot\|_{\infty,\omega}$, $\|\cdot\|_{\omega}$ and $\langle\cdot,\cdot\rangle_\omega$.

In a natural way, the Leberger space $L^p(J;X)~(1\le p< \infty)$ with the interval $J$ and Banach space $X$ (here, $X=H^{-1}(\omega)$, $L^2(\omega)$, $H^1_0(\omega)$, or $H^1(\omega)$), is composed of all those functions $u(t)$ which take values in $X$ for almost every $t\in J$ and whose norm  on the space $L^p(J;X)$ is finite, i.e., $$\|u\|_{L^p(J;X)}=\left(\int_J\|u(\cdot,t)\|_X^pdt\right)^{1/p}<\infty.$$
For $p=\infty$, $L^\infty(J;X)$ is the Banach space of (classes of) measurable functions from $J$ into $X$, which is essentially bounded and whose norm is denoted by
$$\|u\|_{L^{\infty}(J;X)}:=\esssup_{t\in J} \|u(\cdot,t)\|_X.$$

Then, corresponding to the elliptic operators $\mathcal A$, we introduce the bilinear forms $a(\cdot,\cdot)$, which is  defined on
$H^1_0(\Omega)\times H^1_0(\Omega)$ by
$$a(\varphi,\chi):=\langle A\nabla\varphi,\nabla\chi\rangle,\qquad\forall \varphi,\chi\in H^1_0(\Omega).$$
It is not difficult to prove that the bilinear forms $a(\cdot,\cdot)$ is coercive and continuous on $H^1_0(\Omega)$, i.e.,
 \begin{eqnarray*}
 a(\varphi,\varphi)\ge \eta \|\varphi\|^2_1,\quad |a(\varphi,\chi)|\le \beta \|\varphi\|_1\|\chi\|_1,\quad \forall \varphi,\chi\in H^1_0(\Omega)
 \end{eqnarray*}
 with $\eta,\beta\in \mathbb R^+$.

With the above notations and assumptions in hand, the problem (\ref{eq1.1})-(\ref{eq1.3}) can be therefore written in weak form: suppose $f\in C(0,T;L^{2}(\Omega))$, find $u\in C^1(0,T; H^1_0(\Omega))$ such that
\begin{eqnarray}\label{eq2.1}
\langle \partial^\alpha_tu,\varphi\rangle +a(u(t),\varphi)&=&\langle f,\varphi\rangle, ~~~\forall \varphi\in H^1_0(\Omega),~t\in (0,T],\\
u(\cdot,0)&=&u_0(\cdot).\label{eq2.2a}
\end{eqnarray}

In order to study the properties of fractional derivatives,  we introduce the following two lemmas, which are the cornerstones of the theoretical analysis.
\begin{lemma} \label{lem1}
Let $\alpha \in (0,1)$ and $g\in C(0,T;H)\bigcap W^{1,\infty}(\delta,t;H)$, where the limit $ \delta \rightarrow 0^+$ holds and $H$ is a Hilbert space with inner product $(\cdot,\cdot)$ and norm $||\cdot||$. Then
\begin{eqnarray*}
\int_0^T(\partial^\alpha_tg(t),g(t))dt \geq &\frac{1}{2\Gamma(1-\alpha)}\int_0^T ((T-t)^{-\alpha}+t^{-\alpha})||g(t)||^2dt\\
&-\frac{1}{\Gamma(1-\alpha)}\int_0^Tt^{-\alpha}(g(0),g(t))dt.
\end{eqnarray*}
\end{lemma}
\begin{proof}
The line of proof of this lemma mainly follows the Lemma 4.1 in \cite{Banjai22}, but the conclusion still holds after relaxing the regularity assumption from $g\in C^1(0,T;H)$ to $g\in C(0,T;H)\bigcap W^{1,\infty}(\delta,t;H)$. It is worth noting that the relaxed regularity assumption is more appropriate for our problem. Then set $H=R$, and let $ \eta>0 $ and consider
 \begin{equation}\label{1.4}
\begin{aligned}
&\int_0^T\int_0^t(t-\tau+\eta)^{-\alpha}g'(\tau)g(t)d\tau dt= \int_0^T\int_0^\delta(t-\tau+\eta)^{-\alpha}g'(\tau)g(t)d\tau dt \\
&\qquad-\alpha \int_0^T\int_\delta^t(t-\tau+\eta)^{-\alpha-1}g(\tau)g(t)d\tau dt+ \eta^{-\alpha}\int_0^Tg^2(t)dt-g(\delta)\int_0^T(t-\delta+\eta)^{-\alpha}g(t)dt.
\end{aligned}
\end{equation}
Similar to the process of the proof of Lemma 4.1 in \cite{Banjai22}, the right three terms of equation (\ref{1.4}) lead to the following inequality:
 \begin{equation}\label{1.5}
\begin{aligned}
&-\alpha \int_0^T\int_\delta^t(t-\tau+\eta)^{-\alpha-1}g(\tau)g(t)d\tau dt+ \eta^{-\alpha}\int_0^Tg^2(t)dt-g(\delta)\int_0^T(t-\delta+\eta)^{-\alpha}g(t)dt \\
&\quad \geq\frac{1}{2}\int_0^{T}(T-\tau-\eta)^{-\alpha}g^2(\tau)d\tau+\frac{1}{2}\int_0^T(t-\delta+\eta)^{-\alpha}g^2(t)dt-g(\delta)\int_0^T(t-\delta+\eta)^{-\alpha}g(t)dt\\
&\quad \quad+\frac{1}{2}\int_0^\delta \eta^{-\alpha}g^2(\tau)d\tau.
\end{aligned}
\end{equation}
Taking the limits $ \eta \rightarrow 0^+$ and $ \delta \rightarrow 0^+$,  we can obtain the following inequality from (\ref{1.5})
 \begin{equation}\label{1.6}
\begin{aligned}
&-\alpha \int_0^T\int_\delta^t(t-\tau+\eta)^{-\alpha-1}g(\tau)g(t)d\tau dt+ \eta^{-\alpha}\int_0^Tg^2(t)dt-g(\delta)\int_0^T(t-\delta+\eta)^{-\alpha}g(t)dt \\
&\quad \geq\frac{1}{2\Gamma(1-\alpha)}\int_0^T ((T-t)^{-\alpha}+t^{-\alpha})||g(t)||^2dt-\frac{1}{\Gamma(1-\alpha)}\int_0^Tt^{-\alpha}(g(0),g(t))dt.
\end{aligned}
\end{equation}
At the same time, when $\delta$ tends to 0, then $g(\tau) = g(0)~(\tau \in [0,\delta])$ for the first term on the right-hand side of equation (\ref{1.4}), so that $g'(\tau) = 0~(\tau \in [0,\delta])$ holds, and then the term $\int_0^T\int_0^\delta(t-\tau+\eta)^{-\alpha}g'(\tau)g(t)d\tau dt$ also tends to 0, thus we complete the proof of the Lemma \ref{lem1}.
\end{proof}

\begin{lemma} [\cite{Banjai22}] \label{lem2}
For any fixed $\theta \in (0,\pi/2)$, and let $u$ be the solution of (\ref{eq1.1}) with the additional assumption $u_0\in D(\mathcal A)$. Then the bound
\begin{eqnarray*}
||u(t)||\leq ||u_0||+\frac{1}{\sin\theta}C_{\alpha,\varphi}\int_0^t(t-\tau)^{\alpha-1}||f(\tau)-\mathcal Au_0||d\tau
\end{eqnarray*}
holds for $t>0$, where $\varphi=\max(0,\pi-(\pi-\theta)/\alpha)$ and $C_{\alpha,\varphi}=\frac{1}{\pi}(\cos\varphi)^{\alpha-1}\Gamma(1-\alpha).$
\end{lemma}

\subsection{Finite element discretization} In this paper, we will use the finite element method for the spatial discretization. Let $0=t_0<t_1<\ldots<t_N=T$ be a partition of the interval $[0,T]$ with $I_n:=(t_{n-1},t_n]$, $k_n=t_n-t_{n-1}$ and $k:=\max_{1\le n\le N}k_n$.  Then corresponding to the time node $t_n$, a family $\{\mathcal T_n\}^N_{n=0}$ of conforming shape-regular triangulations of the domain $\Omega$ is introduced. It is worth noting that the triangulations here may change over time.  Let
$$h_n(x)=h_K:={\hbox{diam}}(K),\qquad K \in \mathcal T_n\quad {\hbox{and}} \quad x\in K,$$
denote the local mesh-size function of each given triangulation $\mathcal T_n$ and let $h:=\max_{0\le n\le N}h_n$. For each $n$ and each $K\in \mathcal T_n$, let $\overline \partial_n K$ be the set of the sides of $K$ (edges in $d=2$ or faces in $d=3$) and $\partial_n K\subset \overline\partial_n K$ be the set of the internal sides of $K$, i.e., the sides that do not belong to the boundary of $\Omega$. In addition, the sets $\overline \Sigma_n:=\bigcup_{K\in \mathcal T_n}\overline\partial_n K$ and $\Sigma_n:=\bigcup_{K\in \mathcal T_n}\partial_n K$ are also introduced. Therefore, corresponding to $\mathcal T_n$, we associate the following finite element space:
$$V^n=\left\{v_h\in H^1(\Omega)|v_{h|K}\in \mathcal P_r, ~\forall K\in \mathcal T_n\right\},$$
where $\mathcal P_r$ denotes the space of polynomials of degree
at most $r$ with $r\in \mathbb Z^+$. Furthermore, we set
\begin{eqnarray*}
V^n_0=V^n\cap H^1_0.
\end{eqnarray*}

Following \cite{Lakkis06}, for a function $v\in V^n_0$, the bilinear form $a(v,\varphi)$ can be represented as
\begin{eqnarray}\label{eq2.11}
a(v,\varphi)=\sum\limits_{K\in \mathcal T_n}\langle -{\rm div}(A\nabla v),\varphi\rangle_K+\sum\limits_{\partial_n K\in \Sigma_n}\langle J_A[v],\varphi\rangle_{\partial_n K},
\end{eqnarray}
where $J_A[v]$ is the spatial jump of the field $A\nabla v$ across an element side $\partial_n K\in \Sigma_n$ defined as
\begin{eqnarray*}
 &&J_A[v]|_{\partial_n K}(x)=[A\nabla v]_{\partial_n K}(x):=\lim\limits_{\varepsilon\to 0}(A\nabla v(x+\varepsilon \nu_K)-A\nabla v(x-\varepsilon \nu_K))\cdot \nu_K,
\end{eqnarray*}
where $\nu_K$ is a unit normal vector to $\partial_n K$ at the point $x$. For a finite element function $v\in V^n_0$, denote by $\mathcal A_{el}$ the regular part of the distribution $-{\rm div}(A\nabla v)$, which is defined as a piecewise continuous function such that
\begin{eqnarray*}
\langle \mathcal A_{el}v,\varphi\rangle =\sum\limits_{K\in \mathcal T_n}\langle -{\rm div}(A\nabla v),\varphi\rangle_K,\qquad \forall \varphi\in H^1_0(\Omega).
\end{eqnarray*}
Thus, the representation (\ref{eq2.11}) can be written in the shorter form,
\begin{eqnarray*}
a(v,\varphi)=\langle \mathcal A_{el}v,\varphi\rangle +\langle J_A[v],\varphi\rangle_{\Sigma_n},\qquad \forall \varphi\in H^1_0(\Omega).
\end{eqnarray*}

With the help of the bilinear form, we introduce the discrete elliptic operators in the following Definition, which is the basis of elliptic space reconstruction.
\begin{definition}[Discrete elliptic operators, \cite{Wang21}] The discrete elliptic operator associated with the bilinear form $a$ and the finite element space $V^n_0$ is the operator ${\mathcal A^n_h}:H^1_0\to V^n_0$ defined by
\begin{equation*}
\langle {\mathcal A^n_h}v,\varphi_h\rangle=a(v,\varphi_h), \qquad \forall \varphi_h\in V^n_0.
\end{equation*}
\end{definition}

At the end of this section, based on the theoretical analysis framework of the finite element method, we  introduce the $L^2(\Omega)$-projection operator and the Cl\'ement-type interpolant. Among them,  the $L^2(\Omega)$-projection operator $P^n_0:L^2(\Omega)\to V^n_0$ satisfies
$$\langle P^n_0v-v,\varphi_h\rangle=0,\qquad \forall \varphi_h\in V^n_0,$$
and the stability property and the approximation properties of the Cl\'ement-type interpolant are studied in the following Proposition.

\begin{proposition}[\cite{Lakkis06,Clement75,Scott90,Brenner08}]\label{pro2.2} Let $\Pi^n_h:H^1_0(\Omega)\to V_0^n$ be the  Cl\'ement-type interpolant. Then
\begin{eqnarray*}
\|\Pi^n_hv\|_1\le C_\Pi\|v\|_1.
\end{eqnarray*}
Further, for $j\le r+1$, the following approximation properties are satisfied:
\begin{eqnarray*}
\|h^{-j}_n(v-\Pi^n_hv)\|\le C_{\Pi,j}\|v\|_j,\\
\|h^{1/2-j}_n(v-\Pi^n_hv)\|_{\Sigma_n}\le C_{\Pi,j}\|v\|_j,
\end{eqnarray*}
where $r$ is the finite element polynomial degree, and the constants $C_{\Pi}$ and $C_{\Pi,j}$ depend only on the shape regularity \cite{Brenner08} of the family of triangulations $\{\mathcal T_n\}^N_{n=0}$.
\end{proposition}

\section{Fully discrete approximations to TFPDEs}
In order to describe the time discretization corresponding to (\ref{eq2.1})-(\ref{eq2.2a}), we have introduced a time mesh $0=t_0<t_1<\ldots<t_N=T$. For $n=1,2,\ldots,N$, we define
$$\partial v^n:=\frac{v^n-v^{n-1}}{k_n},\qquad \bar\partial v^n:=\frac{v^n-\Pi^n_{n-1} v^{n-1}}{k_n}, $$
where $\Pi^i_j:V^{j}_0\to V^i_0$ denotes suitable projections or interpolants to be chosen.

\subsection{L1 method in time discretization}
The Caputo fractional derivative $\partial^\alpha_t u$
is approximated by the  L1 finite element approximation
\begin{eqnarray}\label{eq2.15}
\bar\partial^\alpha_kU_h^n&=&\sum\limits_{j=1}^{n}\frac{\Pi_j^nU_h^{j}-\Pi_{j-1}^nU_h^{j-1}}{k_{j}}
\int^{t_{j}}_{t_{j-1}}\omega_{1-\alpha}(t_n-s)ds \\
&=&\frac{1}{\Gamma(2-\alpha)}\sum\limits_{j=1}^{n}\Pi_j^n\frac{U_h^{j}-\Pi_{j-1}^jU_h^{j-1}}{k_{j}}
\left[(t_n-t_{j-1})^{1-\alpha}-(t_n-t_{j})^{1-\alpha}\right]\nonumber\\
&=&\sum\limits_{j=1}^{n}a_{j}(t_n)\Pi_j^n\bar{\partial}U_h^j,\nonumber
\end{eqnarray}
where
$$ a_{j}(t):=\frac{1}{\Gamma(2-\alpha)}\left[(t-t_{j-1})^{1-\alpha}-(t-t_{j})^{1-\alpha}\right].$$
\subsection{Linear time reconstructions}
The $L1$  finite element  approximation can be stated as follows: Given $U_h(t)$, find $U^n_h\in V^n_0$ for all $n=1,2,\ldots,N$ such that
\begin{eqnarray}\label{eqBE1}
&&\langle\bar\partial^\alpha_k U^n_h,\varphi_h\rangle +a(U^{n}_h,\varphi_h)
=\langle f^n_h,\varphi_h\rangle ,\qquad \forall \varphi_h\in V^n_0,
\end{eqnarray}
where $f^n_h=P^n_0f(t_n)$ and the continuous time approximation $U_h(t)~(t\in [0,T])$ is defined by linearly interpolating the nodal values $U^{n}_h$ and $U^{n-1}_h$ for $t\in I_n$, $1\le n\le N$, i.e.
\begin{eqnarray}\label{eq2.5}
U_h(t)&:=& \ell_{n,0}(t)U^{n}_h+\ell_{n,-1}(t)U^{n-1}_h=U^{n-1/2}_h+(t-t_{n-1/2})\partial U^{n}_h
\end{eqnarray}
with
\begin{eqnarray*}
\ell_{n,{0}}(t):=\frac{t-t_{n-1}}{k_n},\qquad \ell_{n,{-1}}(t):=\frac{t_{n}-t}{k_n}.
\end{eqnarray*}
The fully discrete scheme ({\ref{eqBE1}) can be written in the following equivalent pointwise form:  Find $U^n_h\in V^n_0$ satisfying
\begin{eqnarray}\label{eqBE}
&&\bar\partial^\alpha_k U^n_h +\mathcal A^n_hU^{n}_h=f^n_h,\quad {\rm for~ all}\quad n=1,2,\ldots,N.
\end{eqnarray}

\subsection{Quadratic time reconstructions}
In \cite{Wang20}, we also introduce a continuous, piecewise reconstruction $\widehat U_h$ of $U_h$ in time defined for all $t\in I_n$, $2\le n\le N$ by
\begin{eqnarray}\label{eq4.6}
\widehat U_h(t)&:=&U_h(t)+\frac{1}{2}(t-t_{n-1})(t-t_{n})\widehat W^{n}_h\\
&=& U^{n-1/2}_h+(t-t_{n-1/2})\partial U^{n}_h+\frac{1}{2}(t-t_{n-1})(t-t_{n})\widehat W^{n}_h,\nonumber
\end{eqnarray}
where
\begin{eqnarray*}
\widehat W^{n}_h:=\bar\partial^2 U^{n}_h:=\frac{2(\bar\partial U^{n}_h-\Pi^n_{n-1}\bar\partial U^{n-1}_h)}{k_{n}+k_{n-1}}\quad {\hbox{for}}~~~ 2\le n\le N,
\end{eqnarray*}
We note that $U_h(t)$ and $\widehat U_h(t)$ coincide at $t_n$, $n=2,\ldots,N$.

\subsection{Elliptic space reconstruction} To derive a posteriori error estimate of optimal order in the $L^\infty(L^2)$-norm or $L^1(L^2)$-norm, we have to appropriately define the elliptic reconstruction operator. And the elliptic reconstruction operator $\mathcal R^n$ has been introduced in \cite{Lakkis06,Makridakis03} for pure parabolic problems.

\begin{definition} [Elliptic reconstruction, \cite{Lakkis06,Makridakis03}] \label{def3.1}
 We define the elliptic reconstruction operator associated with the bilinear form $a$ and the finite element space $V^n_0$ to be the unique operator $\mathcal R^n: V^n_0\to H^1_0$ such that
\begin{equation*}
a(\mathcal R^nv,\varphi)=\langle\mathcal A^n_hv,\varphi\rangle, \qquad \forall \varphi\in H^1_0,
\end{equation*}
for a given $v\in V^n_0$. The function $\mathcal R^nv$ is referred to as the elliptic reconstruction.
\end{definition}
It can easily be seen that the elliptic reconstruction $\mathcal R^n$ satisfies the Galerkin orthogonality property, i.e., $v-\mathcal R^nv$ is orthogonal, with respect to $a$, to $V^n_0$:
\begin{eqnarray*}
a(v-\mathcal R^nv,\varphi_h)=0,\qquad \forall \varphi_h\in V^n_0.
\end{eqnarray*}
This property allows us to obtain the following result whose proof uses standard techniques in a posteriori error estimates for elliptic problems \cite{Brenner08,Verfurth13,Ainsworth00}. Then the error estimates of elliptic space reconstruction  can be bounded in the following lemma, which is an important basis for our computation of the a posteriori error estimate.

\begin{lemma}[Elliptic reconstruction error estimates,  \cite{Bansch12,Lakkis06,Makridakis03}] \label{lem4}
 For any $v\in V^n_0$, the following estimates hold:
\begin{eqnarray*}
\|\mathcal R^nv-v\|&\le& \mathcal E_{e0}(v):=C_{e0}\left(\|h^2_n(\mathcal A_{el}-\mathcal A^n_h)v\|+\|J_A[v]h^{3/2}_n\|_{{\Sigma_n}}\right),\nonumber\\
\|\mathcal R^nv-v\|_1&\le& \mathcal E_{e1}(v):=C_{e1}\left(\|h_n(\mathcal A_{el}-\mathcal A^n_h)v\|+\|J_A[v]h^{1/2}_n\|_{{\Sigma_n}}\right),\nonumber
\end{eqnarray*}
where the constants $C_{e0}$ and $C_{e1}$ depend only on $C_{\Pi,1}$, $C_{\Pi,2}$, the domain $\Omega$, and the problem data.
\end{lemma}

\section{A posteriori error estimate of  the fully discrete scheme}

\subsection{A posteriori error estimate for the linear space-time reconstructions}

To derive optimal $L^\infty(L^2)$ and $L^1(L^2)$ a posteriori error estimates of the fully discrete scheme for the L1 method, we  introduce the space-time reconstructions of the discrete solutions, which will be required to derive the error estimates.

Combining the linear reconstruction and elliptic reconstruction operators, we define the linear space-time reconstruction $U_{\mathcal R}:[0,T]\to H^1_0$ as
\begin{equation}\label{eq3.3}
U_{\mathcal R}(t):=\ell_{n,-1}(t)\mathcal R^{n-1}U_h^{n-1}+\ell_{n,0}(t)\mathcal R^{n} U_h^n,\quad {\hbox{for}}~t\in I_n.
\end{equation}
Since $U_{\mathcal R}$ is a linear function on $I_n$, we have $U_{\mathcal R}^\prime(t)=\frac{\mathcal R^{n}U_h^n-\mathcal R^{n-1}U_h^{n-1}}{k_n}$. Then for $t\in I_{n}$, one gets the fractional derivative of $U_{\mathcal R}$ as follow
\begin{eqnarray}\label{eq3.4}
\partial^\alpha_tU_{\mathcal R}(t)&=&\int^{t}_{0}\omega_{1-\alpha}(t-s)U_{\mathcal R}^\prime(s)ds\\
&=&\frac{1}{\Gamma(1-\alpha)}\sum\limits_{j=1}^{n-1}\int^{t_{j}}_{t_{j-1}}(t-s)^{-\alpha}U_{\mathcal R}^\prime(s)ds
+\frac{1}{\Gamma(1-\alpha)}\int^{t}_{t_{n-1}}(t-s)^{-\alpha}U_{\mathcal R}^\prime(s)ds\nonumber\\
&=&\sum\limits_{j=1}^{n-1}a_j(t)\frac{\mathcal R^{j}U_h^j-\mathcal R^{j-1}U_h^{j-1}}{k_j}+\frac{(t-t_{n-1})^{1-\alpha}}{\Gamma(2-\alpha)}\frac{\mathcal R^{n}U_h^n-\mathcal R^{n-1}U_h^{n-1}}{k_n}.\nonumber
\end{eqnarray}
And then the residual of $U_{\mathcal R}$ can be defined as
\begin{eqnarray}\label{eq3.5}
E_1^n(t)=\partial^\alpha_t U_{\mathcal R}(t)+\mathcal AU_{\mathcal R}(t)-f(t)\in H_0^1, \qquad t\in I_n.
\end{eqnarray}
Further the error $\rho:=u-U_{\mathcal R}$ satisfies the error equation
\begin{eqnarray}\label{eq3.6}
\partial^\alpha_t \rho(t)+\mathcal A\rho(t)=-E_1^n(t).
\end{eqnarray}

Hence, based on the error equation (\ref{eq3.6}) and Lemma \ref{lem1}, the $L^1(0,T;L^2(\Omega))$ norm  error estimates  for the linear space-time reconstruction are stated in the following theorem.
\begin{theorem}[$L^1(0,T;L^2(\Omega))$ norm  error estimates] \label{thm1}
 Let $C_{\alpha,T}=\frac{\Gamma(1-\alpha)T^{\alpha+1}}{2^\alpha}$, $U_{\mathcal R}(t)$ be the L1 approximation to the solution of the problem (\ref{eq1.1})-(\ref{eq1.3}) and the error $\rho(t):=u(t)-U_{\mathcal R}(t)$. Then the following a posteriori error estimate is valid, that is,
\begin{equation*}
\begin{aligned}
\int_0^T||\rho(t)||dt+\sqrt{\frac{8C_{\alpha,T}}{T}}&\int_0^T\|\rho(t)\|_1dt \leq 8\sqrt{\frac{T^{1-\alpha}C_{\alpha,T}}{\Gamma(2-\alpha)}}||\rho(0)||+8\sqrt{2}C_1C_{\alpha,T}\int_{0}^{T}\|\bar\partial^\alpha_k U^n_h-\partial^\alpha_t U_{\mathcal R}\|dt\\
&\quad+4\beta\sqrt{C_{\alpha,T}}\left(\int_0^T\|\mathcal R^n U_h^n-U_{\mathcal R}\|^2_1dt\right)^{1/2}+8\sqrt{T}C_{\alpha,T}\left(\int_{0}^{T}\|f-f_h^n\|^2dt\right)^{1/2}.
\end{aligned}
\end{equation*}
\end{theorem}
\begin{proof}
Taking the inner product of the equation (\ref{eq3.6}) with $\rho$ and integrating it in the time direction, we can obtain
\begin{eqnarray}\label{eq3.8}
\int_0^T(\partial^\alpha_t \rho,\rho)dt+\int_0^T(\mathcal A\rho,\rho)dt=\int_{0}^{T}(-E_1^n,\rho)dt.
\end{eqnarray}
By means of Lemma \ref{lem1} and (\ref{eq3.8}), we have
\begin{equation*}
\begin{aligned}
\frac{1}{2\Gamma(1-\alpha)}\int_0^T((T-t)^{-\alpha}&+t^{-\alpha})||\rho(t)||^2dt+\int_0^T||\rho(t)||_1^2dt \leq
\int_{0}^{T}(-E_1^n,\rho)dt+\frac{1}{1-\alpha}\int_0^Tt^{-\alpha}(\rho(0),\rho(t))dt\\
&\leq \int_{0}^{T}(-E_1^n,\rho)dt+ \frac{1}{\Gamma(1-\alpha)}\int_0^T((T-t)^{-\alpha}+t^{-\alpha})||\rho(0)||\cdot||\rho(t)||dt\\
&\leq \int_{0}^{T}(-E_1^n,\rho)dt+ \frac{1}{\Gamma(1-\alpha)}\int_0^T((T-t)^{-\alpha}+t^{-\alpha})(||\rho(0)||^2+\frac{1}{4}||\rho(t)||^2)dt.
\end{aligned}
\end{equation*}
Since the bilinear forms is continuous, by rearranging the above equation and using the fact $(T-t)^{-\alpha}+t^{-\alpha}\geq (\frac{T}{2})^{-\alpha}$, we obtain
\begin{equation}\label{eq3.9}
\begin{aligned}
\frac{2^\alpha}{4\Gamma(1-\alpha)T^\alpha}&\int_0^T||\rho(t)||^2dt+\int_0^T||\rho(t)||_1^2dt \leq
\int_{0}^{T}(-E_1^n,\rho)dt +\frac{1}{\Gamma(1-\alpha)}\int_0^T((T-t)^{-\alpha}+t^{-\alpha})dt||\rho(0)||^2\\
&\leq \int_{0}^{T} (-\partial^\alpha_t U_{\mathcal R}-\mathcal AU_{\mathcal R}+f+\bar\partial^\alpha_k U^n_h +\mathcal A^n_hU^{n}_h-f^n_h,\rho)dt+\frac{2T^{1-\alpha}}{\Gamma(2-\alpha)}||\rho(0)||^2\\
&\leq \int_{0}^{T}\|\bar\partial^\alpha_k U^n_h-\partial^\alpha_t U_{\mathcal R}\|\cdot\|\rho\|dt+\beta\int_0^T\|{\mathcal R}^n U_h^n-U_{\mathcal R}\|_1\cdot\|\rho(t)\|_1dt\\
&\quad +\int_0^T\|f-f_h^n\|\cdot\|\rho\|dt+\frac{2T^{1-\alpha}}{\Gamma(2-\alpha)}||\rho(0)||^2\\
&\leq C_1\int_{0}^{T}\|\bar\partial^\alpha_k U^n_h-\partial^\alpha_t U_{\mathcal R}\|dt\cdot\int_0^T\|\rho\|dt+\frac{\beta^2}{2}\int_0^T\|{\mathcal R}^nU_h^n-U_{\mathcal R}\|^2_1dt+\frac{1}{2}\int_0^T\|\rho(t)\|^2_1dt\\
&\quad+\frac{\Gamma(1-\alpha)T^\alpha}{2^{\alpha-1}}\int_{0}^{T}\|f-f_h^n\|^2dt+\frac{2^\alpha}{8\Gamma(1-\alpha)T^\alpha}\int_0^T\|\rho(t)\|^2dt
+\frac{2T^{1-\alpha}}{\Gamma(2-\alpha)}||\rho(0)||^2.
\end{aligned}
\end{equation}
Sorting out the similar terms from (\ref{eq3.9}), we arrive at
\begin{equation*}
\begin{aligned}
\frac{2^\alpha}{8\Gamma(1-\alpha)T^\alpha}\int_0^T&||\rho(t)||^2dt +\frac{1}{2}\int_0^T\|\rho(t)\|^2_1dt \leq C_1\int_{0}^{T}\|\bar\partial^\alpha_k U^n_h-\partial^\alpha_t U_{\mathcal R}\|dt\cdot\int_0^T\|\rho\|dt\\
&+\frac{\beta^2}{2}\int_0^T\|{\mathcal R}^nU_h^n-U_{\mathcal R}\|^2_1dt+\frac{\Gamma(1-\alpha)T^\alpha}{2^{\alpha-1}}\int_{0}^{T}\|f-f_h^n\|^2dt+\frac{2T^{1-\alpha}}{\Gamma(2-\alpha)}||\rho(0)||^2.
\end{aligned}
\end{equation*}
Further, recalling the fact $(\int_0^T\|\rho(t)\|dt)^2 \leq T\int_0^T\|\rho(t)\|^2dt$ and using the Young's inequality to the above inequality, we have
\begin{equation*}
\begin{aligned}
\frac{1}{8C_{\alpha,T}}\left(\int_0^T||\rho(t)||dt\right)^2 &+\frac{1}{2T}\left(\int_0^T\|\rho(t)\|_1dt\right)^2 \leq 4C_1^2C_{\alpha,T}\left(\int_{0}^{T}\|\bar\partial^\alpha_k U^n_h-\partial^\alpha_t U_{\mathcal R}\|dt\right)^2\\
&\quad+\frac{1}{16C_{\alpha,T}}\left(\int_0^T||\rho(t)||dt\right)^2+\frac{\beta^2}{2}\int_0^T\|\mathcal R^n U_h^n-U_{\mathcal R}\|^2_1dt\\
&\quad+2TC_{\alpha,T}\int_{0}^{T}\|f-f_h^n\|^2dt+\frac{2T^{1-\alpha}}{\Gamma(2-\alpha)}||\rho(0)||^2.
\end{aligned}
\end{equation*}
Then using the properties of the sum of squares yields
\begin{equation*}
\begin{aligned}
\int_0^T||\rho(t)||dt+\sqrt{\frac{8C_{\alpha,T}}{T}}&\int_0^T\|\rho(t)\|_1dt \leq 8\sqrt{\frac{T^{1-\alpha}C_{\alpha,T}}{\Gamma(2-\alpha)}}||\rho(0)||+8\sqrt{2}C_1C_{\alpha,T}\int_{0}^{T}\|\bar\partial^\alpha_k U^n_h-\partial^\alpha_t U_{\mathcal R}\|dt\\
&\quad+4\beta\sqrt{C_{\alpha,T}}\left(\int_0^T\|\mathcal R^n U_h^n-U_{\mathcal R}\|^2_1dt\right)^{1/2}+8\sqrt{T}C_{\alpha,T}\left(\int_{0}^{T}\|f-f_h^n\|^2dt\right)^{1/2}.
\end{aligned}
\end{equation*}
Therefore, we give the proof of Theorem \ref{thm1}.
\end{proof}

\begin{remark}
 For the sake of formal intuition and conciseness, we give the expressions of a posteriori error estimate in Theorem \ref{thm1}. But since the elliptic reconstruction operators defined in Definition \ref{def3.1} is a continuous abstract operator, the terms $\int_{0}^{T}\|\bar\partial^\alpha_k U^n_h-\partial^\alpha_t U_{\mathcal R}\|dt$ and $\int_0^T\|U_h^n-U_{\mathcal R}\|^2_1dt$ cannot be computed directly. Therefore, with the help of Lemma \ref{lem4}, we will give their computable expressions in Appendices 1 and 2 respectively.
\end{remark}

For the linear space-time reconstruction, from the proof of Theorem \ref{thm1}, it can be seen that a posteriori error estimate in the $L^1(0,T;L^2(\Omega))$-norm contains the results of a posteriori error estimate in the $L^2(0,T;L^2(\Omega))$-norm, which will be given in this corollary below.
\begin{corollary}[$L^2(0,T;L^2(\Omega))$ norm error estimate]  \label{cor1}
 Let $U_{\mathcal R}(t)$ be the L1 approximation to the solution of the problem (\ref{eq1.1})-(\ref{eq1.3}), and the error $\rho(t):=u(t)-U_{\mathcal R}(t)$. Then the following a posteriori error estimate is valid, that is,
\begin{equation*}
\begin{aligned}
\int_0^T||\rho(t)||^2dt&\leq \frac{2^{4-\alpha}\Gamma(1-\alpha)T}{\Gamma(2-\alpha)}||\rho(0)||^2+ 2^{5-2\alpha}\Gamma^2(1-\alpha)C_1^2T^{2\alpha+1}\left(\int_{0}^{T}\|\bar\partial^\alpha_k U^n_h-\partial^\alpha_t U_{\mathcal R}\|dt\right)^2\\
&+2^{1-\alpha}\Gamma(1-\alpha)\beta^2T^\alpha\int_0^T\|{\mathcal R}^nU_h^n-U_{\mathcal R}\|^2_1dt+2^{5-2\alpha}\Gamma^2(1-\alpha)T^{2\alpha}\int_{0}^{T}\|f-f_h^n\|^2dt.
\end{aligned}
\end{equation*}
\end{corollary}
\begin{proof}
From (\ref{eq3.9}), it can be deduced that
\begin{equation*}
\begin{aligned}
\frac{2^\alpha}{4\Gamma(1-\alpha)T^\alpha}&\int_0^T||\rho(t)||^2dt+\int_0^T||\rho(t)||_1^2dt
\leq C_1\int_{0}^{T}\|\bar\partial^\alpha_k U^n_h-\partial^\alpha_t U_{\mathcal R}\|dt\cdot\int_0^T\|\rho\|dt\\
&\quad+\beta\int_0^T\|{\mathcal R}^n U_h^n-U_{\mathcal R}\|_1\cdot\|\rho(t)\|_1dt+\int_0^T\|f-f_h^n\|\cdot\|\rho\|dt+\frac{2T^{1-\alpha}}{\Gamma(2-\alpha)}||\rho(0)||^2.
\end{aligned}
\end{equation*}
Recalling the fact $(\int_0^T\|\rho(t)\|dt)^2 \leq T\int_0^T\|\rho(t)\|^2dt$ and using the Young's inequality to the above inequality, we have
\begin{equation*}
\begin{aligned}
\frac{2^\alpha}{4\Gamma(1-\alpha)T^\alpha}&\int_0^T||\rho(t)||^2dt+\int_0^T||\rho(t)||_1^2dt \leq \frac{C_1^2}{4\eta_1}\left(\int_{0}^{T}\|\bar\partial^\alpha_k U^n_h-\partial^\alpha_t U_{\mathcal R}\|dt\right)^2\\
&+\eta_1 T \int_0^T\|\rho\|^2dt+\frac{\beta ^2}{4\eta_2}\int_0^T\|{\mathcal R}^nU_h^n-U_{\mathcal R}\|^2_1dt+\eta_2\int_0^T\|\rho(t)\|^2_1dt\\
&+\frac{1}{4\eta_3}\int_{0}^{T}\|f-f_h^n\|^2dt+\eta_3\int_0^T||\rho(t)||^2dt+\frac{2T^{1-\alpha}}{\Gamma(2-\alpha)}||\rho(0)||^2.
\end{aligned}
\end{equation*}
Further, taking $\eta_1=\frac{2^\alpha}{16\Gamma(1-\alpha)T^{\alpha+1}},~\eta_2=1,~\eta_3=\frac{2^\alpha}{16\Gamma(1-\alpha)T^{\alpha}}$ and merging the same terms yields
\begin{equation*}
\begin{aligned}
\frac{2^\alpha}{8\Gamma(1-\alpha)T^\alpha}&\int_0^T||\rho(t)||^2dt\leq \frac{2T^{1-\alpha}}{\Gamma(2-\alpha)}||\rho(0)||^2+ 2^{2-\alpha}\Gamma(1-\alpha)C_1^2T^{\alpha+1}\left(\int_{0}^{T}\|\bar\partial^\alpha_k U^n_h-\partial^\alpha_t U_{\mathcal R}\|dt\right)^2\\
&+\frac{\beta ^2}{4}\int_0^T\|{\mathcal R}^nU_h^n-U_{\mathcal R}\|^2_1dt+2^{2-\alpha}\Gamma(1-\alpha)T^{\alpha}\int_{0}^{T}\|f-f_h^n\|^2dt.
\end{aligned}
\end{equation*}
Therefore, by normalizing the coefficients of the above inequality, we could complete the proof of Corollary \ref{cor1}.
\end{proof}

Meanwhile, applying the Lemma \ref{lem2} into the error equation (\ref{eq3.6}), we obtain the following $L^\infty(0,T;L^2(\Omega))$ norm  error estimates  for the linear space-time reconstruction.
\begin{theorem}[$L^\infty(0,T;L^2(\Omega))$ norm  error estimates] \label{thm2}
 Let $U_{\mathcal R}(t)$ be the L1 approximation to the solution of the problem (\ref{eq1.1})-(\ref{eq1.3}), and the error $\rho(t):=u(t)-U_{\mathcal R}(t)$. Then the following a posteriori error estimate is valid, that is, for $t\in I_n$, $n\ge 1$,
\begin{equation*}
\begin{aligned}
||\rho(t)||\leq ||\rho(0)||+\frac{1}{\sin\theta}C_{\alpha,\varphi}\int_0^t(t-\tau)^{\alpha-1}(||\partial^\alpha_t U_{\mathcal R}-\bar\partial^\alpha_kU_h^n||+||\mathcal AU_{\mathcal R}-\mathcal A _h^nU_h^n||+||f_h^n-f||+||\mathcal A\rho(0)||)d\tau.
\end{aligned}
\end{equation*}
\end{theorem}
\begin{proof}
Applying Lemma \ref{lem2} to the error equation (\ref{eq3.6}) yields
\begin{equation*}
\begin{aligned}
||\rho(t)||&\leq ||\rho(0)||+\frac{1}{\sin\theta}C_{\alpha,\varphi}\int_0^t(t-\tau)^{\alpha-1}||E_1^n+A\rho(0)||d\tau\\
&\leq ||\rho(0)||+\frac{1}{\sin\theta}C_{\alpha,\varphi}\int_0^t(t-\tau)^{\alpha-1}||\partial^\alpha_t U_{\mathcal R}(t)+\mathcal AU_{\mathcal R}(t)-f(t)+\mathcal A\rho(0)||d\tau\\
&\leq ||\rho(0)||+\frac{1}{\sin\theta}C_{\alpha,\varphi}\int_0^t(t-\tau)^{\alpha-1}||(\partial^\alpha_t U_{\mathcal R}-\bar\partial^\alpha_kU_h^n) +(\mathcal AU_{\mathcal R}-\mathcal A _h^nU_h^n)+(f_h^n-f)+\mathcal A\rho(0)||d\tau\\
&\leq ||\rho(0)||+\frac{1}{\sin\theta}C_{\alpha,\varphi}\int_0^t(t-\tau)^{\alpha-1}(||\partial^\alpha_t U_{\mathcal R}-\bar\partial^\alpha_kU_h^n||+||\mathcal AU_{\mathcal R}-\mathcal A _h^nU_h^n||+||f_h^n-f||+||\mathcal A\rho(0)||)d\tau.
\end{aligned}
\end{equation*}
Thus, for the linear space-time reconstruction, a posteriori error estimate  in the sense of $L^\infty(0,T;L^2(\Omega))$-norm is proved.
\end{proof}

\subsection{A posteriori error estimate for the quadratic space-time reconstruction}

 In this subsection, we begin by introducing the quadratic space-time reconstruction to derive a posteriori error estimates in the $L^\infty(L^2)$-norm and $L^1(L^2)$-norm for $L1$ schemes.

Let $U^n_h\in V^n_0$ denote the discrete approximate solution produced by the L1 scheme (\ref{eq2.15}), and the linear space-time reconstruction $U_{\mathcal R}:[0,T]\to H^1_0$  has been defined in (\ref{eq3.3}). For $t\in I_1$, we introduce the fractional integral reconstruction, i.e.,
\begin{eqnarray}\label{eq5.1}
\widehat{U}_{\mathcal R}(t)&=&\widehat U_{\mathcal R}(0)-\partial^{-\alpha}_t \mathcal AU_{\mathcal R}(t)+\partial^{-\alpha}_t \widehat f_h^n\\
&=&U_{\mathcal R}(0)-\int_{0}^{t}\omega_\alpha(t-s)\mathcal AU_{\mathcal R}(s)ds+\int_{0}^{t}\omega_\alpha(t-s)\widehat f_h^n(s)ds\nonumber\\
&=&U_{\mathcal R}(t)-t\frac{\mathcal R^1 U_h^1-\mathcal R^0 U_h^0}{k_1}+ \frac{t^\alpha}{\Gamma(\alpha+1)}(f_h^0- \mathcal A \mathcal R^0U_h^0 )\nonumber\\
&&+\frac{t^{\alpha+1}}{\Gamma(\alpha+2)}\left(\frac{f_h^1-f_h^0}{k_1}-\frac{\mathcal A ( \mathcal R^1U_h^1  -\mathcal R^0U_h^0)}{k_1}\right),\nonumber
\end{eqnarray}
where $\widehat f_h^n(t)=\ell_{n,-1}(t)f_h^0+\ell_{n,0}(t)f_h^1$.

From (\ref{eq5.1}), it is easy to obtain the following pointwise equation
\begin{eqnarray}\label{eq4.2}
\partial^\alpha_t\widehat{U}_{\mathcal R}(t)+\mathcal AU_{\mathcal R}(t)=\widehat f_h^n(t),\qquad \forall t\in I_{1}.
\end{eqnarray}

From (\ref{eq1.1}) and (\ref{eq4.2}), we know that for $t\in I_{1}$, it holds that
\begin{eqnarray}\label{eq4.4}
\partial^\alpha_t\widehat \rho(t)+\mathcal A \rho(t)=-E_2^n(t)
\end{eqnarray}
with $\widehat \rho(t)=u(t)-\widehat U_{\mathcal R}(t)$ and $E_2^n(t)=\widehat f_h^n(t)-f(t)$.

 Then let $\widehat U_{\mathcal R}:[0,T]\to H^1_0$ denote the quadratic space-time reconstructions, given by
\begin{eqnarray}\label{eq4.1}
\widehat U_{\mathcal R}(t):=U_{\mathcal R}(t)+\frac{1}{2}(t-t_{n-1})(t-t_n)\mathcal R^n\widehat W^n_h
\end{eqnarray}
 where $t\in I_n$ for each $n\in\{2,\ldots,N\}$.

From the definition of the Caputo fractional derivative and the quadratic space-time reconstructions (\ref{eq4.1}), it follows that
\begin{eqnarray*}
\partial^\alpha_t\widehat U_{\mathcal R}(t)&=&\int^{t}_{0}\omega_{1-\alpha}(t-s)\widehat U_{\mathcal R}^\prime(s)ds\nonumber\\
&=&\partial^\alpha_t U_{\mathcal R}(t)+\sum\limits_{j=2}^{n-1}{\mathcal R}^j \widehat W_h^j\int^{t_{j}}_{t_{j-1}}\omega_{1-\alpha}(t-s)(s-t_{j-1/2})ds\nonumber\\
&&+{\mathcal R}^n\widehat W_h^n\int^{t}_{t_{n-1}}\omega_{1-\alpha}(t-s)(s-t_{n-1/2})ds\nonumber\\
&&+\int^{t_{1}}_{0}\omega_{1-\alpha}(t-s)\left(\frac{s^\alpha}{\Gamma(\alpha+1)}\left(\frac{f_h^1-f_h^0}{k_1}-\frac{\mathcal A ( \mathcal R^1U_h^1  -\mathcal R^0U_h^0)}{k_1}\right)+\frac{s^{\alpha-1}}{\Gamma(\alpha)}(f_h^0-{\mathcal R}^0 \mathcal A U_h^0)\right)ds\nonumber\\
&&-\frac{\mathcal R^1 U_h^1-\mathcal R^0 U_h^0}{\Gamma(2-\alpha)k_1}\left(t^{1-\alpha}-(t-t_1)^{1-\alpha}\right).
\end{eqnarray*}
Using integration by parts, we further obtain
\begin{eqnarray}\label{eq4.7}
\partial^\alpha_t\widehat U_{\mathcal R}(t)=\partial^\alpha_t U_{\mathcal R}(t)+\mathcal W_h^n,
\end{eqnarray}
where $\mathcal W_h^n$ is defined by
\begin{eqnarray}\label{eq4.8}
\mathcal W_h^n&=&-\frac{1}{2\Gamma(2-\alpha)}\sum\limits_{j=2}^{n-1}{\mathcal R}^j\widehat W_h^j
k_j\left[(t-t_j)^{1-\alpha}+(t-t_{j-1})^{1-\alpha}\right]\nonumber\\
&&-\frac{1}{\Gamma(3-\alpha)}\sum\limits_{j=2}^{n-1}{\mathcal R}^j\widehat W_h^j\left[(t-t_j)^{2-\alpha}-(t-t_{j-1})^{2-\alpha}\right]\\
&&-\frac{{\mathcal R}^n\widehat W_h^nk_n}{2\Gamma(2-\alpha)}(t-t_{n-1})^{1-\alpha}+\frac{{\mathcal R}^n\widehat W_h^n}{\Gamma(3-\alpha)}(t-t_{n-1})^{2-\alpha}\nonumber\\
&&+\int^{t_{1}}_{0}\omega_{1-\alpha}(t-s)\left(\frac{s^\alpha}{\Gamma(\alpha+1)}\left(\frac{f_h^1-f_h^0}{k_1}-\frac{\mathcal A ( \mathcal R^1U_h^1  -\mathcal R^0U_h^0)}{k_1}\right)+\frac{s^{\alpha-1}}{\Gamma(\alpha)}(f_h^0-{\mathcal R}^0 \mathcal A U_h^0)\right)ds\nonumber\\
&&-\frac{\mathcal R^1 U_h^1-\mathcal R^0 U_h^0}{\Gamma(2-\alpha)k_1}\left(t^{1-\alpha}-(t-t_1)^{1-\alpha}\right).\nonumber
\end{eqnarray}
Let $\widehat \rho=u-\widehat U_{\mathcal R}$. Then it follows from (\ref{eq1.1}) and (\ref{eq4.7}) that
\begin{eqnarray}\label{eq4.9}
\partial^\alpha_t\widehat \rho(t)+\mathcal A\rho(t)=-E_2^n(t),\qquad n\ge 2,
\end{eqnarray}
where $E_2^n(t)$ is defined by
\begin{eqnarray*}
E_2^n(t)&=&E_1^n(t)+\mathcal W_h^n\nonumber\\
&=&\partial^\alpha_t U_{\mathcal R}(t)+\mathcal AU_{\mathcal R}(t)-f(t)+\mathcal W_h^n\nonumber\\
&=&(\partial^\alpha_t U_{\mathcal R}-\bar\partial^\alpha_kU_h^n) +(\mathcal AU_{\mathcal R}-\mathcal A_h^nU_h^n)+(f_h^n-f)+\mathcal W_h^n.\nonumber\\
\end{eqnarray*}

Therefore, from (\ref{eq4.4}) and (\ref{eq4.9}), we can obtain a unified form of the error equation for the quadratic space-time reconstruction,  namely
\begin{eqnarray}\label{eq4.10}
\partial^\alpha_t\widehat \rho(t)+\mathcal A\rho(t)=-E_2^n(t),\qquad n\ge 1.
\end{eqnarray}

To analyze the contribution of each term to a posteriori error estimate, we divide the residual $E_2^n(t)$ into the following four terms:

\begin{equation*}
E_f^n=\left \{
\begin{aligned}
&f-\hat f_h^n,~&n=1,\\
&f- f_h^n,~&n\geq2;
\end{aligned}
 \right.
\quad E_{\mathcal W}^n=\left \{
\begin{aligned}
&0,~&n=1,\\
&\mathcal W_h^n,~&n\geq2;
\end{aligned}
 \right.
\end{equation*}
\begin{equation*}
E_{\alpha}^n=\left \{
\begin{aligned}
&0,~&n=1,\\
&\partial^\alpha_t U_{\mathcal R}-\bar\partial^\alpha_kU_h^n,~&n\geq2;
\end{aligned}
 \right.
\quad E_I^n=\left \{
\begin{aligned}
&0,~&n=1,\\
&\mathcal AU_{\mathcal R}-\mathcal A_h^nU_h^n,~&n\geq2.
\end{aligned}
 \right.
\end{equation*}

Using Lemma \ref{lem1}, the following $L^1(0,T;L^2(\Omega))$ norm  error bounds for the quadratic space-time reconstruction can be formulated from the error equation (\ref{eq4.10}).

\begin{theorem}[$L^1(0,T;L^2(\Omega))$ norm error estimate]\label{thm3}
  Let $C_{\alpha,T}=\frac{\Gamma(1-\alpha)T^{\alpha+1}}{2^\alpha}$, and $U_{\mathcal R}(t)$ be the L1 approximation to the solution of the problem (\ref{eq1.1})-(\ref{eq1.3}), $\widehat U_{\mathcal R}$ be the corresponding reconstruction of $U_{\mathcal R}$, $\rho=u-U_{\mathcal R}$ and $\widehat \rho=u-\widehat U_{\mathcal R}$. Then the following a posteriori error estimate is valid, that is,
\begin{equation*}
\begin{aligned}
\int_0^T||\widehat \rho(t)||dt&+\sqrt{\frac{2C_{\alpha,T}}{T}}\int_0^T\|\widehat{\rho}(t)\|_1dt+\sqrt{\frac{4C_{\alpha,T}}{T}}\int_0^T\|\rho(t)\|_1dt\leq
\sqrt{\frac{48C_{\alpha,T}T^{1-\alpha}}{\Gamma(2-\alpha)}}||\widehat \rho(0)||\\
&\quad+4\sqrt{3}C_1C_{\alpha,T}\int_0^T\|E_{ f}^n(t)+E_{\alpha}^n(t)+E_{\mathcal W}^n(t)\|dt+2\sqrt{3C_{\alpha,T}}\left(\int_0^T\|U_{\mathcal R}(t)-\widehat U_{\mathcal R}(t)\|_1^2dt\right)^{1/2}\\
&\quad+2\beta\sqrt{6C_{\alpha,T}}\left(\int_0^T\|\mathcal R ^n U_h^n-U_{\mathcal R}\|_1dt\right)^{1/2}.
\end{aligned}
\end{equation*}
\end{theorem}
\begin{proof}
Taking the inner product of the equation (\ref{eq4.9}) with $\widehat \rho$ and integrating it on the interval $[0,T]$, we can obtain
\begin{equation}\label{eq5.12}
\begin{aligned}
\int_0^T(\partial^\alpha_t\widehat \rho(t),\widehat \rho(t))dt+\int_0^T(\mathcal A\rho(t),\widehat \rho(t))dt=\int_0^T(-E_2^n(t),\widehat \rho(t))dt.
\end{aligned}
\end{equation}
Using the fact
\begin{eqnarray*}
(\mathcal A\rho(t),\widehat\rho(t))&=&\frac{1}{2}\left(\|\widehat{\rho}(t)\|^{2}_1+\|\rho(t)\|^{2}_1-\|\widehat{\rho}(t)-\rho(t)\|^{2}_1\right)\nonumber\\
&=&\frac{1}{2}\left(\|\widehat{\rho}(t)\|^{2}_1+\|\rho(t)\|^{2}_1-\|U_{\mathcal R}(t)-\widehat U_{\mathcal R}(t)\|^{2}_1\right),
\end{eqnarray*}
and applying the Lemma \ref{lem1} to the equation (\ref{eq5.12}), we get
\begin{equation}\label{eq4.12}
\begin{aligned}
&\frac{1}{2\Gamma(1-\alpha)}\int_0^T((T-t)^{-\alpha}+t^{-\alpha})||\widehat \rho(t)||^2dt+\frac{1}{2}\int_0^T\left(\|\widehat{\rho}(t)\|^{2}_1+\|\rho(t)\|^{2}_1\right)dt\leq \frac{1}{2}\int_0^T\|U_{\mathcal R}(t)-\widehat U_{\mathcal R}(t)\|^{2}_1dt\\
&\quad\quad+   \int_0^T(-E_f^n(t)-E_I^n(t)-E_{\alpha}^n(t)-E_{\mathcal W}^n(t),\widehat \rho(t))dt+ \frac{1}{\Gamma(1-\alpha)}\int_0^Tt^{-\alpha}(\widehat \rho(0),\widehat \rho(t))dt\\
&\quad \leq C_1\int_0^T\|-E_f^n(t)-E_{\alpha}^n(t)-E_{\mathcal W}^n(t)\|dt\cdot\int_0^T\|\widehat {\rho}(t)\|dt+\frac{1}{4\Gamma(1-\alpha)}\int_0^T((T-t)^{-\alpha}+t^{-\alpha})||\widehat \rho(t)||^2dt\\
&\quad\quad+\frac{||\widehat \rho(0)||^2}{\Gamma(1-\alpha)}\int_0^T((T-t)^{-\alpha}+t^{-\alpha})dt+\frac{1}{2}\int_0^T\|U_{\mathcal R}(t)-\widehat U_{\mathcal R}(t)\|^{2}_1dt+\int_0^T(-E_I^n(t),\widehat \rho(t))dt.
\end{aligned}
\end{equation}
Since the bilinear forms is continuous, combining the same terms and recalling the fact $(T-t)^{-\alpha}+t^{-\alpha})\geq (\frac{T}{2})^{-\alpha}$  in the above inequality yield
\begin{equation*}
\begin{aligned}
&\frac{2^\alpha}{4\Gamma(1-\alpha)T^\alpha}\int_0^T||\widehat \rho(t)||^2dt+\frac{1}{2}\int_0^T\left(\|\widehat{\rho}(t)\|^{2}_1+\|\rho(t)\|^{2}_1\right)dt\leq \frac{2T^{1-\alpha}}{\Gamma(2-\alpha)}||\widehat \rho(0)||^2+\beta^2\int_{t_1}^T\|\mathcal R ^n U_h^n-U_{\mathcal R}\|_1^2dt\\
&\qquad+\frac{1}{4}\int_0^T\|\widehat \rho(t)\|_1^2  dt+C_1\int_0^T\|E_f^n(t)+E_{\alpha}^n(t)+E_{\mathcal W}^n(t)\|dt\cdot\int_0^T\|\widehat {\rho}(t)\|dt
+\frac{1}{2}\int_0^T\|U_{\mathcal R}(t)-\widehat U_{\mathcal R}(t)\|^{2}_1dt.
\end{aligned}
\end{equation*}
Then using the inequality $(\int_0^T\|\rho(t)\|dt)^2 \leq T\int_0^T\|\rho(t)\|^2dt$ and the Young's inequalities, we further obtain
\begin{equation*}
\begin{aligned}
\frac{1}{4C_{\alpha,T}} \left(\int_0^T||\widehat \rho(t)||dt\right)^2&+\frac{1}{4T}\left(\int_0^T\|\widehat{\rho}(t)\|_1dt\right)^2+\frac{1}{2T}\left(\int_0^T\|\rho(t)\|_1dt\right)^2\leq \frac{2T^{1-\alpha}}{\Gamma(2-\alpha)}||\widehat \rho(0)||^2 \\
&\quad+\beta^2\int_{t_1}^T\|\mathcal R ^n U_h^n-U_{\mathcal R}\|_1^2dt+2C_1^2C_{\alpha,T}\left(  \int_0^T\|E_f^n(t)\|+\|E_{\alpha}^n(t)\|+\|E_{\mathcal W}^n(t\|dt \right)^2 \\
&\quad+\frac{1}{8C_{\alpha,T}}\left( \int_0^T||\widehat \rho(t)||dt \right)^2+ \frac{1}{2}\int_0^T\|U_{\mathcal R}(t)-\widehat U_{\mathcal R}(t)\|^{2}_1dt.
\end{aligned}
\end{equation*}
Therefore, using the properties of the sum of squares, we complete the proof of the Theorem \ref{thm3}.
\end{proof}

\begin{remark}
 By means of Lemma \ref{lem4}, we could obtain the computable expressions of the terms $\int_0^T \|U_{\mathcal R}(t)-\widehat U_{\mathcal R}(t)\|_1^2dt$,$~\int_0^T\|E_{\mathcal I}^n(t)\|dt$,$~\int_0^T\|E_{\mathcal W}^n(t)\|dt$ in Appendix 3-5, respectively.
 \end{remark}

As for the case of the quadratic space-time reconstruction, similar to the conclusion of the linear space-time reconstruction in Corollary \ref{cor2}, a posteriori error estimate in  the $L^2(0,T;L^2(\Omega))$-norm can also be deduced from Theorem \ref{thm3}, and the details are given in the following corollary.

\begin{corollary}[$L^2(0,T;L^2(\Omega))$ norm error estimate] \label{cor2}
Let $U_{\mathcal R}(t)$ be the L1 approximation to the solution of the problem (\ref{eq1.1})-(\ref{eq1.3}), $\widehat U_{\mathcal R}$ be the corresponding reconstruction of $U_{\mathcal R}$, $\rho=u-U_{\mathcal R}$ and $\widehat \rho=u-\widehat U_{\mathcal R}$. Then the following a posteriori error estimate is valid, that is,
\begin{equation*}
\begin{aligned}
 \int_0^T||\widehat \rho(t)||^2dt&+\frac{4C_{\alpha,T}}{T}\int_0^T\|\rho(t)\|^{2}_1dt\leq \frac{16C_{\alpha,T}}{\Gamma(2-\alpha)T^{\alpha}}||\widehat \rho(0)||^2
+\frac{4\beta^2C_{\alpha,T}}{T}\int_{t_1}^T\|\mathcal R ^n U_h^n-U_{\mathcal R}\|_1^2dt\\
&\quad+\frac{16C_1^2C_{\alpha,T}^2}{T}\left(  \int_0^T\|E_f^n(t)\|+\|E_{\alpha}^n(t)\|+\|E_{\mathcal W}^n(t\|dt \right)^2+\frac{4C_{\alpha,T}}{T}\int_0^T\|U_{\mathcal R}(t)-\widehat U_{\mathcal R}(t)\|^{2}_1dt.
\end{aligned}
\end{equation*}
\end{corollary}
\begin{proof}
Recalling the fact $(T-t)^{-\alpha}+t^{-\alpha})\geq (\frac{T}{2})^{-\alpha}$, we start the proof of Corollary \ref{cor2} from (\ref{eq4.12}) in Theorem \ref{thm3}.
\begin{equation*}
\begin{aligned}
&\frac{2^\alpha}{4\Gamma(1-\alpha)T^\alpha}\int_0^T||\widehat \rho(t)||^2dt+\frac{1}{2}\int_0^T\left(\|\widehat{\rho}(t)\|^{2}_1+\|\rho(t)\|^{2}_1\right)dt\leq \frac{2T^{1-\alpha}}{\Gamma(2-\alpha)}||\widehat \rho(0)||^2+\frac{\beta^2}{2}\int_{t_1}^T\|\mathcal R ^n U_h^n-U_{\mathcal R}\|_1^2dt\\
&\qquad+\frac{1}{2}\int_0^T\|\widehat \rho(t)\|_1^2  dt+C_1\int_0^T\|E_f^n(t)+E_{\alpha}^n(t)+E_{\mathcal W}^n(t)\|dt\cdot\int_0^T\|\widehat {\rho}(t)\|dt
+\frac{1}{2}\int_0^T\|U_{\mathcal R}(t)-\widehat U_{\mathcal R}(t)\|^{2}_1dt.
\end{aligned}
\end{equation*}
Then using the inequality $(\int_0^T\|\rho(t)\|dt)^2 \leq T\int_0^T\|\rho(t)\|^2dt$ and the Young's inequalities, we further obtain
\begin{equation*}
\begin{aligned}
\frac{T}{4C_{\alpha,T}} \int_0^T||\widehat \rho(t)||^2dt&+\frac{1}{2}\int_0^T\|\rho(t)\|^{2}_1dt\leq \frac{2T^{1-\alpha}}{\Gamma(2-\alpha)}||\widehat \rho(0)||^2
+\frac{\beta^2}{2}\int_{t_1}^T\|\mathcal R ^n U_h^n-U_{\mathcal R}\|_1^2dt\\
&\quad+2C_1^2C_{\alpha,T}\left(  \int_0^T\|E_f^n(t)\|+\|E_{\alpha}^n(t)\|+\|E_{\mathcal W}^n(t\|dt \right)^2+\frac{T}{8C_{\alpha,T}} \int_0^T||\widehat \rho(t)||^2dt\\
&\quad+\frac{1}{2}\int_0^T\|U_{\mathcal R}(t)-\widehat U_{\mathcal R}(t)\|^{2}_1dt.
\end{aligned}
\end{equation*}
Combining the same terms in the above equation and normalizing the coefficients yields
\begin{equation*}
\begin{aligned}
 \int_0^T||\widehat \rho(t)||^2dt&+\frac{4C_{\alpha,T}}{T}\int_0^T\|\rho(t)\|^{2}_1dt\leq \frac{16C_{\alpha,T}}{\Gamma(2-\alpha)T^{\alpha}}||\widehat \rho(0)||^2
+\frac{4\beta^2C_{\alpha,T}}{T}\int_{t_1}^T\|\mathcal R ^n U_h^n-U_{\mathcal R}\|_1^2dt\\
&\quad+\frac{16C_1^2C_{\alpha,T}^2}{T}\left(  \int_0^T\|E_f^n(t)\|+\|E_{\alpha}^n(t)\|+\|E_{\mathcal W}^n(t\|dt \right)^2+\frac{4C_{\alpha,T}}{T}\int_0^T\|U_{\mathcal R}(t)-\widehat U_{\mathcal R}(t)\|^{2}_1dt.
\end{aligned}
\end{equation*}
Thus for the quadratic space-time reconstruction, the $L^2(0,T;L^2(\Omega))$-norm error estimate is stated in Corollary \ref{cor2}.
\end{proof}

Then we demonstrate   a posteriori error estimate in the sense of $L^\infty(0,T;L^2(\Omega))$-norm for the quadratic space-time reconstruction in the following theorem.
\begin{theorem}[$L^\infty(0,T;L^2(\Omega))$ norm error estimate] \label{thm4}
 Let $U_{\mathcal R}(t)$ be the L1 approximation to the solution of the problem (\ref{eq1.1})-(\ref{eq1.3}), $\widehat U_{\mathcal R}$ be the corresponding reconstruction of $U_{\mathcal R}$, $\rho=u-U_{\mathcal R}$ and $\widehat \rho=u-\widehat U_{\mathcal R}$. Then the following a posteriori error estimate is valid, for $t\in I_n$, $n\ge 1$,
\begin{equation*}
\begin{aligned}
||\widehat \rho(t)||&\leq ||\widehat \rho(0)||+\frac{1}{\sin\theta}C_{\alpha,\varphi}\int_0^t(t-\tau)^{\alpha-1}(||E_f^n(\tau)||+||E_{\alpha}^n(\tau)||+||E_{\mathcal W}^n(\tau)||))d\tau \\
&\quad+ \frac{1}{\sin\theta}C_{\alpha,\varphi}\int_0^t(t-\tau)^{\alpha-1}||\mathcal A_h^n U_h^n-\mathcal A\widehat U_{\mathcal R}||d\tau
+\frac{1}{\sin\theta}C_{\alpha,\varphi}\int_0^t(t-\tau)^{\alpha-1}||\mathcal A\widehat \rho(0)||d\tau.
\end{aligned}
\end{equation*}
\end{theorem}
\begin{proof}
Firstly, we rewrite equation (\ref{eq4.9}) in the following form
\begin{eqnarray}\label{eq5.16}
\partial^\alpha_t\widehat \rho(t)+\mathcal A\widehat \rho(t)=-E_2^n(t)+\mathcal A\left(U_{\mathcal R}(t)-\widehat U_{\mathcal R}(t)\right).
\end{eqnarray}
Then applying the Lemma \ref{lem2} to (\ref{eq5.16}) gives
\begin{equation*}
\begin{aligned}
||\widehat \rho(t)||&\leq ||\widehat \rho(0)||+\frac{1}{\sin\theta}C_{\alpha,\varphi}\int_0^t(t-\tau)^{\alpha-1}||-E_f^n(\tau)-E_{\alpha}^n(\tau)-E_{\mathcal W}^n(\tau)+(\mathcal A_h^n U_h^n-\mathcal A\widehat U_{\mathcal R}(\tau))-\mathcal A\widehat \rho(0)||d\tau\\
&\leq ||\widehat \rho(0)||+\frac{1}{\sin\theta}C_{\alpha,\varphi}\int_0^t(t-\tau)^{\alpha-1}(||E_f^n(\tau)||+||E_{\alpha}^n(\tau)||+||E_{\mathcal W}^n(\tau)||))d\tau \\
&\quad+ \frac{1}{\sin\theta}C_{\alpha,\varphi}\int_0^t(t-\tau)^{\alpha-1}||\mathcal A_h^n U_h^n-\mathcal A\widehat U_{\mathcal R}(\tau))||d\tau
+\frac{1}{\sin\theta}C_{\alpha,\varphi}\int_0^t(t-\tau)^{\alpha-1}||\mathcal A\widehat \rho(0)||d\tau.
\end{aligned}
\end{equation*}
Therefore, we complete the proof of  a posteriori error estimate in the $L^\infty(0,T;L^2(\Omega))$-norm.
\end{proof}

\section{Final error estimates of the fully discrete scheme}
In this section, with the help of triangular inequality and  aforementioned a posteriori error estimates, we are devoted to giving optimal global final error estimates  and pointwise final error estimates for the linear space-time reconstruction and quadratic space-time reconstruction.
\subsection{Final error estimates for the linear space-time reconstructions}
In this subsection, we will investigate the final error estimates for the linear space-time reconstruction. To do this, we divide the final error estimates into several terms and control each one separately.
\begin{eqnarray*}
\|u-U_h\|\leq \|u-U_{\mathcal R}\|+\|U_{\mathcal R}-U_h\|.
\end{eqnarray*}
Since the term $\|u-U_{\mathcal R}\|$ has already been discussed as a posteriori error estimate, we only need to analyze the term $\|U_{\mathcal R}-U_h\|$. For $t\in I_n$, by the definition of the linear reconstruction (\ref{eq2.5}) and the linear space-time reconstruction (\ref{eq3.3}) and using Lemma \ref{lem4}, we can obtain
\begin{eqnarray}\label{eq5.3}
\|U_{\mathcal R}-U_h\| &=& \| \ell_{n,0}(t)(\mathcal R^n U^{n}_h-U^{n}_h)+\ell_{n,-1}(t)(\mathcal R^{n-1}U^{n-1}_h-U^{n-1}_h) \|\\
&\leq& max\{ \|\mathcal R^n U^{n}_h-U^{n}_h\|,\|\mathcal R^{n-1}U^{n-1}_h-U^{n-1}_h\| \}\nonumber \\
&\leq& max\{ \mathcal E_{e0}(U^{n}_h),\mathcal E_{e0}(U^{n-1}_h)\}\nonumber.
\end{eqnarray}

Therefore, for the linear space-time reconstruction,  global and pointwise final error estimates can be deduced  from the triangular inequality and a posteriori error estimate in the following two theorems, respectively.
\begin{theorem}[A global final error estimate]\label{thm5}
 Let $U_h(t),~U_{\mathcal R}(t)$ be the L1 approximation to the solution of the problem (\ref{eq1.1})-(\ref{eq1.3}), and the error $\rho(t):=u(t)-U_{\mathcal R}(t)$. Then the following a posteriori error estimate is valid, that is,
\begin{equation*}
\begin{aligned}
\int_0^T||u-U_h||dt &\leq T\max\limits_{n=0}^N \mathcal E_{e0}(U^{n}_h)+8\sqrt{\frac{T^{1-\alpha}C_{\alpha,T}}{\Gamma(2-\alpha)}}||\rho(0)||+8\sqrt{2}C_1C_{\alpha,T}\int_{0}^{T}\|\bar\partial^\alpha_k U^n_h-\partial^\alpha_t U_{\mathcal R}\|dt\\
&\quad+4\beta\sqrt{C_{\alpha,T}}\left(\int_0^T\|\mathcal R^n U_h^n-U_{\mathcal R}\|^2_1dt\right)^{1/2}+8\sqrt{T}C_{\alpha,T}\left(\int_{0}^{T}\|f-f_h^n\|^2dt\right)^{1/2}.
\end{aligned}
\end{equation*}
\end{theorem}
\begin{proof}
By means of (\ref{eq5.3}) and Theorem \ref{thm1}, a global final error estimate is  formulated and the proof of Theorem \ref{thm5} is completed.
\end{proof}

\begin{theorem}[A pointwise final error estimate] \label{thm6}
 Let $U_h(T),~U_{\mathcal R}(t)$ be the L1 approximation to the solution of the problem (\ref{eq1.1})-(\ref{eq1.3}), and the error $\rho(t):=u(t)-U_{\mathcal R}(t)$. Then the following a posteriori error estimate is valid, that is, for $t\in I_n$, $n\ge 1$,
\begin{equation*}
\begin{aligned}
||u-U_h||&\leq  max\{ \mathcal E_{e0}(U^{n}_h),\mathcal E_{e0}(U^{n-1}_h)\}+||\rho(0)||\\
&\quad+\frac{1}{\sin\theta}C_{\alpha,\varphi}\int_0^t(t-\tau)^{\alpha-1}(||\partial^\alpha_t U_{\mathcal R}-\bar\partial^\alpha_kU_h^n||+||\mathcal AU_{\mathcal R}-\mathcal A _h^nU_h^n||+||f_h^n-f||+||\mathcal A\rho(0)||)d\tau.
\end{aligned}
\end{equation*}
\end{theorem}
\begin{proof}
Similarly, from (\ref{eq5.3}) and Theorem \ref{thm2}, a pointwise final error estimate can be derived and the proof of Theorem \ref{thm6} can be completed.
\end{proof}

\subsection{Final error estimates for the quadratic space-time reconstruction}
In this subsection, for the quadratic space-time reconstruction, we are committed to giving its a global final error estimate and a pointwise final error estimate. In order to achieve this goal, the final error for the quadratic space-time reconstruction can be divided into several parts as follow:
\begin{eqnarray}\label{eq5.2}
\|u-U_h\|\leq \|u-\widehat U_{\mathcal R}\|+\|\widehat U_{\mathcal R}- \widehat U_h \|+\|\widehat U_h-U_h\|.
\end{eqnarray}
Since the first term on the right-hand side of (\ref{eq5.2}) is discussed in the Section 4, our task becomes to control the last two terms. Then in view of equation (\ref{eq4.1}) and Lemma \ref{lem4}, the elliptic  reconstruction error $\|\widehat U_{\mathcal R}- \widehat U_h \|$ can be estimated by
\begin{eqnarray}\label{eq5.5}
\|\widehat U_{\mathcal R}- \widehat U_h \| \leq \mathcal E_{e0}(\widehat U_h ).
\end{eqnarray}
Meanwhile, using the fact that $(t-t_{n-1})(t-t_n)\leq \frac{k_n^2}{4}$, the time reconstruction error $\|\widehat U_h-U_h\|$ can be bounded from (\ref{eq4.6}) by
\begin{eqnarray}\label{eq5.6}
\|\widehat U_h-U_h\| \leq \frac{k_n^2}{8}\|\widehat W^{n}_h\|.
\end{eqnarray}

Therefore, from the above results combined with  Theorem \ref{thm3} and Theorem \ref{thm4}, we can obtain  a global final error estimate and pointwise error estimate for the quadratic space-time reconstruction in the following two theorems, respectively.
\begin{theorem}[A global final error estimate]\label{thm7}
 Let $U_{\mathcal R}(t)$ be the L1 approximation to the solution of the problem (\ref{eq1.1})-(\ref{eq1.3}), $\widehat U_{\mathcal R}$ be the corresponding reconstruction of $U_{\mathcal R}$, $\rho=u-U_{\mathcal R}$ and $\widehat \rho=u-\widehat U_{\mathcal R}$. Then the following a posteriori error estimate is valid, that is,
\begin{equation*}
\begin{aligned}
\int_0^T||u-U_h ||dt&\leq \int_0^T \mathcal E_{e0}(\widehat U_h )dt + \sum\limits_{n=1}^N \frac{k_n^3}{8} \|\widehat W^{n}_h\|+\sqrt{\frac{48C_{\alpha,T}T^{1-\alpha}}{\Gamma(2-\alpha)}}||\widehat \rho(0)||\\
&\quad+4\sqrt{3}C_1C_{\alpha,T}\int_0^T\|E_{ f}^n(t)+E_{\alpha}^n(t)+E_{\mathcal W}^n(t)\|dt+2\sqrt{3C_{\alpha,T}}\left(\int_0^T\|U_{\mathcal R}(t)-\widehat U_{\mathcal R}(t)\|_1^2dt\right)^{1/2}\\
&\quad+2\beta\sqrt{6C_{\alpha,T}}\left(\int_0^T\|\mathcal R ^n U_h^n-U_{\mathcal R}\|_1dt\right)^{1/2}.
\end{aligned}
\end{equation*}
\end{theorem}
\begin{proof}
In view of the elliptic reconstruction error, the time reconstruction error and Theorem \ref{thm3}, we obtain a global final error estimate and complete the proof of Theorem \ref{thm7}.
\end{proof}

\begin{theorem}[A pointwise final error estimate] \label{thm8}
 Let $U_h(t),~U_{\mathcal R}(t)$ be the L1 approximation to the solution of the problem (\ref{eq1.1})-(\ref{eq1.3}), $\widehat U_{\mathcal R}$ be the corresponding reconstruction of $U_{\mathcal R}$, $\rho=u-U_{\mathcal R}$ and $\widehat \rho=u-\widehat U_{\mathcal R}$. Then the following a posteriori error estimate is valid, that is, for $t\in I_n$, $n\ge 1$,
\begin{equation*}
\begin{aligned}
||u-U_h||&\leq \mathcal E_{e0}(\widehat U_h )+ \frac{k_n^2}{8}\|\widehat W^{n}_h\|+||\widehat \rho(0)||+\frac{1}{\sin\theta}C_{\alpha,\varphi}\int_0^t(t-\tau)^{\alpha-1}(||E_f^n(\tau)||+||E_{\alpha}^n(\tau)||+||E_{\mathcal W}^n(\tau)||))d\tau \\
&\quad+ \frac{1}{\sin\theta}C_{\alpha,\varphi}\int_0^t(t-\tau)^{\alpha-1}||\mathcal A_h^n U_h^n-\mathcal A\widehat U_{\mathcal R}||d\tau
+\frac{1}{\sin\theta}C_{\alpha,\varphi}\int_0^t(t-\tau)^{\alpha-1}||\mathcal A\widehat \rho(0)||d\tau.
\end{aligned}
\end{equation*}
\end{theorem}
\begin{proof}
Similarly, from (\ref{eq5.5}), (\ref{eq5.6}) and Theorem \ref{thm4}, we prove the Theorem \ref{thm8} and derive a pointwise final error estimate for the quadratic space-time reconstruction.
\end{proof}

From Theorems \ref{thm5}-\ref{thm8}, it can be seen that the final error estimates for either linear space-time reconstruction or quadratic space-time reconstruction does not depend on the true solution of the equation (\ref{eq1.1}), which is very suitable for certain problems where the true solution is unavailable. This is also the original purpose of introducing a posteriori error estimate.

\section{Numerical experiment}\label{sec:6}

In this section, we will give some numerical experiments to investigate the behavior of a posteriori error estimate. Since the theoretical results in the $L^\infty(L^2)$-norm sense are similar to our previous results in semi-discrete scheme (see \cite{CJL22}), we only show the results in the $L^1(L^2)$-norm sense here for the sake of brevity of expression.

\subsection*{Example 6.1 (Smooth solution)}
 Let us consider the following model problem on $\Omega=(0,1)$,
\begin{equation*}
\left\{
\begin{aligned}
&\partial _t^{\alpha} u =\frac{\partial ^2 u}{\partial x^2}+f, \qquad x\in \Omega,\quad 0\le t\le 1,\\
&u(0,t)=u(1,t)=0, \quad 0\le t\le 1,\\
&u(x,0)=u_0(x)=x(1-x).
\end{aligned}
\right.
\end{equation*}
We prescribe the exact solution of the problem as $u(x,t)=(1+t^2)x(1-x)$. Then the corresponding source function is $f(x,t)=\frac{\Gamma(3)}{\Gamma(3-\alpha)}t^{2-\alpha}x(1-x)+2(1+t^2).$

Let us define the estimators
$$\mathcal E_u(T)=\int_0^T\|u-U_h\|dt,\quad\mathcal E_t^{\alpha}(T):=\int_{0}^{T}\|\bar\partial^\alpha_k U^n_h-\partial^\alpha_t U_{\mathcal R}\|dt,\quad \mathcal {E}_U(T):=\left(\int_0^T\|{\mathcal R}^nU_h^n-U_{\mathcal R}\|^2_1dt\right)^{1/2},$$
$$ \mathcal E_f(T):=\left(\int_0^T\|f-f_h^n\|^2dt\right)^{1/2}, \quad \mathcal {E}_W(T):=\int_0^T\|\mathcal W_h^n\|dt,\quad \mathcal E_{\widehat U} (T)=\left(\int_0^T\|U_{\mathcal R}(t)-\widehat U_{\mathcal R}(t)\|_1^2dt\right)^{1/2}.$$

The a posteriori error estimators $\mathcal E_t^{\alpha}(T)$, $\mathcal {E}_U(T)$, $\mathcal E_f(T)$, $\mathcal E_{\widehat U} (T)$, $\mathcal E_{W}(T)$ and a global final error as well as their temporal convergence orders are listed in Tables \ref{table1-1} and \ref{table1-2}, respectively. From Table \ref{table1-1}, we observe that  the a posteriori error estimators $\mathcal E_t^{\alpha}(T)$, $\mathcal {E}_U(T)$ are  of order $2-\alpha$ and  $1$, and a global final error $\mathcal E_u(T)$ is  of optimal order $2-\alpha$. And from the composition of the residual $E_1^n$ of the linear reconstruction in (\ref{eq3.5}), the presence of the term $\mathcal {E}_U(T)$ limits the residual $E_1^n$ to only  order 1, which don't match the convergence of the L1 method. Then we could see that the a posteriori error estimators $\mathcal E_f(T)$,~$\mathcal E_{\widehat U} (T)$ and  $\mathcal E_{W}(T)$ in  Table \ref{table1-2} are  of optimal order $2$, $2$ and $2-\alpha$, respectively. Therefore, from the composition of the residual $E_2^n$,  we know that its convergence order is  of order $2-\alpha$, which matches with the optimal convergence of the L1 method. It should be noted that a posteriori error estimators for the quadratic reconstruction in Theorem \ref{thm3} are the same as the posterior error estimators for the linear reconstruction in Theorem \ref{thm1} and will not be repeated here.  Compared with the linear reconstruction, the quadratic reconstruction can improve the convergence of the residuals in the error equation and thus can control a global final error more precisely, which is the reason why we prefer the quadratic reconstruction.

\begin{table}[htp]
\footnotesize
\begin{center}
\caption{Example 6.1: The errors and their convergence orders of a posteriori error estimators in Theorem \ref{thm1}, where $T=1$ and $M=512$.}\label{table1-1} \vskip -2mm
\def\temptablewidth{0.666\textwidth}
 {\rule{\temptablewidth}{0.9pt}}
\begin{tabular*}{\temptablewidth}{|c|c|cc|cc|cc|}
    $\alpha$ &    $N$   &$\mathcal E_u(T)$ &  order & $\mathcal E_t^{\alpha}(T)$ &  order  & $\mathcal {E}_U(T)$ & order  \\
  \cline{1-8}
                 &  16   & 2.9050E-04  & -      & 2.1420E-04  & -      & 3.4005E-03  & -         \\
           0.25  &  32   & 9.2624E-05  & 1.6491 & 6.6942E-05  & 1.6780 & 1.7008E-03  & 0.9995   \\
                 &  64   & 2.8559E-05  & 1.6975 & 2.0737E-05  & 1.6907 & 8.5049E-04  & 0.9998   \\
                 &  128  & 8.2303E-06  & 1.7949 & 6.3774E-06  & 1.7012 & 4.2526E-04  & 0.9999   \\
     \cline{1-8}
                 &  16  & 3.3305E-04  & -      & 5.6401E-04  & -      & 3.4008E-03  & -        \\
           0.5   &  32  & 1.1995E-04  & 1.4733 & 2.0799E-04  & 1.4392 & 1.7009E-03  & 0.9995   \\
                 &  64  & 4.2456E-05  & 1.4984 & 7.5792E-05  & 1.4564 & 8.5050E-04  & 0.9998   \\
                 &  128 & 1.4535E-05  & 1.5464 & 2.7380E-05  & 1.4689 & 4.2526E-04  & 0.9999   \\
     \cline{1-8}
                 &  16  & 9.2400E-04  & -      & 9.9608E-04  & -      & 3.4006E-03  & -        \\
           0.75  &  32  & 3.9077E-04  & 1.2416 & 4.3214E-04  & 1.2048 & 1.7007E-03  & 0.9996   \\
                 &  64  & 1.6439E-04  & 1.2492 & 1.8535E-04  & 1.2212 & 8.5049E-04  & 0.9997   \\
                 &  128 & 6.8686E-05  & 1.2590 & 7.8914E-05  & 1.2319 & 4.2526E-04  & 0.9999   \\
       \end{tabular*}
{\rule{\temptablewidth}{0.95pt}}
\end{center}
\end{table}

\begin{table}[htp]
\footnotesize
\begin{center}
\caption{Example 6.1: The errors and their convergence orders of a posteriori error estimators in Theorem \ref{thm3}, where $T=1$ and $M=512$.}\label{table1-2} \vskip -2mm
\def\temptablewidth{0.666\textwidth}
 {\rule{\temptablewidth}{0.9pt}}
\begin{tabular*}{\temptablewidth}{|c|c|cc|cc|cc|}
    $\alpha$ &    $N$ & $\mathcal E_f(T)$  &  order & $\mathcal E_{\widehat U} (T)$ &  order  & $\mathcal E_{W}(T)$ & order  \\
  \cline{1-8}
                 &  16   &  5.2803E-03   &  -      &  1.4452E-05   &  -      & 3.6566E-03 & -           \\
           0.25  &  32   &  1.3163E-03   &  2.0041 &  3.6160E-06   &  1.9988 & 1.0482E-03 & 1.8026   \\
                 &  64   &  3.2860E-04   &  2.0021 &  9.0417E-07   &  1.9997 & 3.0221E-04 & 1.7943   \\
                 &  128  &  8.2087E-05   &  2.0011 &  2.2605E-07   &  1.9999 & 8.7602E-05 & 1.7865  \\
     \cline{1-8}
                 &  16   &  5.4191E-03   &  -      &  1.4453E-05   &  -      & 5.4921E-03 & -       \\
           0.5   &  32   &  1.3510E-03   &  2.0040 &  3.6161E-06   &  1.9989 & 1.8252E-03 & 1.5893   \\
                 &  64   &  3.3726E-04   &  2.0021 &  9.0418E-07   &  1.9998 & 6.0708E-04 & 1.5881   \\
                 &  128  &  8.4250E-05   &  2.0011 &  2.2605E-07   &  1.9999 & 2.0743E-04 & 1.5493  \\
     \cline{1-8}
                 &  16   &  5.5225E-03   &  -      &  1.4460E-05   &  -      & 9.9215E-03 & -        \\
           0.75  &  32   &  1.3769E-03   &  2.0039 &  3.6168E-06   &  1.9993 & 4.1978E-03 & 1.2409   \\
                 &  64   &  3.4373E-04   &  2.0020 &  9.0425E-07   &  1.9999 & 1.7777E-03 & 1.2397   \\
                 &  128  &  8.5867E-05   &  2.0011 &  2.2606E-07   &  2.0000 & 7.5171E-04 & 1.2417   \\
       \end{tabular*}
{\rule{\temptablewidth}{0.95pt}}
\end{center}
\end{table}

\subsection*{Example 6.2 (Nonsmooth solution)}

In the second experiment, we consider the equation (\ref{eq1.1})-(\ref{eq1.3}) with a nonsmooth exact solution $u(x,t)=(1+t^\alpha)x(1-x)$. Then the corresponding a posteriori error estimators $\mathcal E_t^{\alpha}(T)$, $\mathcal {E}_U(T)$, $\mathcal E_f(T)$, $\mathcal E_{\widehat U} (T)$, $\mathcal E_{W}(T)$ and a global final error as well as their temporal convergence orders are listed in Tables \ref{table1-3} and \ref{table1-4}, respectively. From Table \ref{table1-3}, we observe that  the a posteriori error estimators $\mathcal {E}_U(T)$ are still of optimal order $2-\alpha$, but the  posteriori error estimator $\mathcal E_t^{\alpha}(T)$ and a global final error $\mathcal {E}_u(T)$ is only  of order $1$, which is consistent with the convergence order of the L1 method for solving nonsmooth problems on the uniform mesh. Then we could see that  a posteriori error estimators $\mathcal E_{\widehat U} (T)$ and $\mathcal {E}_f(T)$ in  Table \ref{table1-4} are of optimal order $2$, and a posteriori error estimator  $\mathcal E_{W}(T)$ is only  of  order $1$. Thus, numerical experiments illustrate that for nonsmooth  solution problems or problems with unclear regularity, the optimal convergence order may not be achieved on uniform mesh either using linear or quadratic reconstruction.

\begin{table}[htp]
\footnotesize
\begin{center}
\caption{Example 6.2: The errors and their convergence orders of a posteriori error estimators in Theorem \ref{thm1}, where $T=1$ and $M=512$.}\label{table1-3} \vskip -2mm
\def\temptablewidth{0.666\textwidth}
 {\rule{\temptablewidth}{0.9pt}}
\begin{tabular*}{\temptablewidth}{|c|c|cc|cc|cc|}
    $\alpha$ &   $N$   & $\mathcal E_u(T)$  &  order & $\mathcal E_t^{\alpha}(T)$ &  order  & $\mathcal {E}_U(T)$ & order  \\
  \cline{1-8}
                 &  16   & 4.7707E-04  & -      & 3.6344E-04  & -      & 2.9791E-03  & -      \\
           0.25  &  32   & 2.0844E-04  & 1.1946 & 1.6953E-04  & 1.1002 & 1.4915E-03  & 0.9981 \\
                 &  64   & 9.6126E-05  & 1.1167 & 7.8010E-05  & 1.1198 & 7.4648E-04  & 0.9986 \\
                 &  128  & 4.8442E-05  & 0.9886 & 3.5602E-05  & 1.1317 & 3.7350E-04  & 0.9989 \\
     \cline{1-8}
                 &  16   & 4.0228E-04  & -      & 2.8196E-04  & -      & 2.9592E-03  & -       \\
           0.5   &  32   & 1.8818E-04  & 1.0961 & 1.3598E-04  & 1.0521 & 1.4804E-03  & 0.9991 \\
                 &  64   & 8.8881E-05  & 1.0822 & 6.4692E-05  & 1.0717 & 7.4051E-04  & 0.9994 \\
                 &  128  & 4.4146E-05  & 1.0096 & 3.0530E-05  & 1.0834 & 3.7036E-04  & 0.9996 \\
     \cline{1-8}
                 &  16   & 5.0443E-04  & -      & 1.6481E-04  & -      & 2.9489E-03  & -       \\
           0.75  &  32   & 2.3575E-04  & 1.0974 & 8.2238E-04  & 1.0029 & 1.4747E-03  & 0.9997 \\
                 &  64   & 1.1606E-04  & 1.0237 & 4.0477E-04  & 1.0227 & 7.3744E-04  & 0.9998 \\
                 &  128  & 5.7584E-05  & 1.0111 & 1.9763E-05  & 1.0343 & 3.6875E-04  & 0.9999 \\
       \end{tabular*}
{\rule{\temptablewidth}{0.95pt}}
\end{center}
\end{table}

\begin{table}[htp]
\footnotesize
\begin{center}
\caption{Example 6.2: The errors and their convergence orders of a posteriori error estimators in Theorem \ref{thm3}, where $T=1$ and $M=512$.}\label{table1-4} \vskip -2mm
\def\temptablewidth{0.666\textwidth}
 {\rule{\temptablewidth}{0.9pt}}
\begin{tabular*}{\temptablewidth}{|c|c|cc|cc|cc|}
    $\alpha$ &    $N$ & $\mathcal E_f(T)$  &  order & $\mathcal E_{\widehat U} (T)$ &  order  & $\mathcal E_{W}(T)$ & order  \\
  \cline{1-8}
                 &  16  & 5.2141E-03  & -       &  9.2046E-06   &  -      & 7.7135E-03 & -           \\
           0.25  &  32  & 1.3182E-03  & 1.9838  &  2.3904E-06   &  1.9451 & 3.8592E-03 & 0.9991   \\
                 &  64  & 3.3191E-04  & 1.9897  &  6.1470E-07   &  1.9593 & 1.9219E-03 & 1.0058   \\
                 &  128 & 8.3354E-05  & 1.9935  &  1.5653E-07   &  1.9734 & 9.5457E-04 & 1.0096  \\
     \cline{1-8}
                 &  16  & 5.1926E-03  & -      &  8.0289E-06   &  -      & 7.3875E-03 & -       \\
           0.5   &  32  & 1.3061E-03  & 1.9912 &  2.1135E-06   &  1.9256 & 3.7692E-03 & 0.9708   \\
                 &  64  & 3.2775E-04  & 1.9946 &  5.5150E-07   &  1.9382 & 1.8963E-03 & 0.9911   \\
                 &  128 & 8.2122E-05  & 1.9968 &  1.4227E-07   &  1.9547 & 9.4578E-04 & 1.0036  \\
     \cline{1-8}
                 &  16  & 5.1915E-03  & -      &  9.6178E-06   &  -      & 5.5731E-03 & -        \\
           0.75  &  32  & 1.3012E-03  & 1.9963 &  2.5649E-06   &  1.9068 & 2.9857E-03 & 0.9004   \\
                 &  64  & 3.2577E-04  & 1.9979 &  6.7805E-07   &  1.9194 & 1.5718E-03 & 0.9257   \\
                 &  128 & 8.1512E-05  & 1.9988 &  1.7582E-07   &  1.9473 & 8.1983E-04 & 0.9390   \\
       \end{tabular*}
{\rule{\temptablewidth}{0.95pt}}
\end{center}
\end{table}

\subsection*{Example 6.3 (Graded mesh)}

From the above numerical experiments, we know that the optimal convergence order of numerical algorithm with nonsmooth data couldn't be achieved on the uniform mesh. In this numerical experiments, we still consider the above nonsmooth solution problem but use the graded mesh, where the graded parameter takes $r=\max{\{(2-\alpha)/\alpha,1\}}.$

Then the corresponding a posteriori error estimators $\mathcal E_t^{\alpha}(T)$, $\mathcal {E}_U(T)$, $\mathcal E_f(T)$, $\mathcal E_{\widehat U} (T)$, $\mathcal E_{W}(T)$ and a global final error as well as their temporal convergence orders are listed in Tables \ref{table1-5} and \ref{table1-6}, respectively. From Table \ref{table1-5}, we observe that a global final error and the a posteriori error estimators $\mathcal E_t^{\alpha}(T)$, $\mathcal {E}_U(T)$  are of optimal order $2-\alpha$, $2-\alpha$ and $1$ respectively. And from Table \ref{table1-6}, we could see that the a posteriori error estimators $\mathcal {E}_f(T)$, $\mathcal E_{\widehat U} (T)$ and  $\mathcal E_{W}(T)$  are of optimal order $2$, $2$ and $2-\alpha$, respectively. Therefore, numerical experiments show that for nonsmooth solution problems, either linear or quadratic reconstruction on graded mesh can achieve its optimal convergence order, which enlightens us that a posteriori error estimates should be used on nonuniform mesh as much as possible for nonsmooth solution problems or problems with unclear regularity.

\begin{table}[htp]
\footnotesize
\begin{center}
\caption{Example 6.3: The errors and their convergence orders of a posteriori error estimators in Theorem \ref{thm1}, where $T=1$ and $M=512$.}\label{table1-5} \vskip -2mm
\def\temptablewidth{0.666\textwidth}
 {\rule{\temptablewidth}{0.9pt}}
\begin{tabular*}{\temptablewidth}{|c|c|cc|cc|cc|}
    $\alpha$ &     $N$   & $\mathcal E_u(T)$  &  order  & $\mathcal E_t^{\alpha}(T)$ &  order  & $\mathcal {E}_U(T)$ & order \\
  \cline{1-8}
                 &  16   & 6.8462E-04  & -      & 5.4893E-04  & -      & 2.9791E-03  & -       \\
           0.25  &  32   & 2.1025E-04  & 1.7032 & 1.6910E-04  & 1.6987 & 1.4915E-03  & 0.9981 \\
                 &  64   & 6.3800E-05  & 1.7205 & 5.1588E-05  & 1.7128 & 7.4648E-04  & 0.9986 \\
                 &  128  & 1.8939E-05  & 1.7522 & 1.5591E-05  & 1.7263 & 3.7350E-04  & 0.9989 \\
     \cline{1-8}
                 &  16   & 9.6730E-04  & -      & 5.6088E-04  & -      & 3.0116E-03  & -       \\
           0.5   &  32   & 3.5117E-04  & 1.4618 & 2.1020E-04  & 1.4159 & 1.5098E-03  & 0.9962 \\
                 &  64   & 1.2570E-04  & 1.4822 & 7.7201E-05  & 1.4451 & 7.5635E-04  & 0.9971 \\
                 &  128  & 4.4146E-05  & 1.5096 & 2.7945E-05  & 1.4660 & 3.7874E-04  & 0.9978 \\
     \cline{1-8}
                 &  16   & 5.7994E-04  & -      & 4.5654E-04  & -      & 3.0600E-03  & -       \\
           0.75  &  32   & 2.5289E-04  & 1.1974 & 2.0009E-04  & 1.1901 & 1.5373E-03  & 0.9931 \\
                 &  64   & 1.0828E-04  & 1.2237 & 8.6428E-04  & 1.2111 & 7.7147E-04  & 0.9947 \\
                 &  128  & 4.5498E-05  & 1.2509 & 3.6993E-05  & 1.2242 & 3.8685E-04  & 0.9958 \\
       \end{tabular*}
{\rule{\temptablewidth}{0.95pt}}
\end{center}
\end{table}

\begin{table}[htp]
\footnotesize
\begin{center}
\caption{Example 6.3: The errors and their convergence orders of a posteriori error estimators in Theorem \ref{thm3}, where $T=1$ and $M=512$.}\label{table1-6} \vskip -2mm
\def\temptablewidth{0.666\textwidth}
 {\rule{\temptablewidth}{0.9pt}}
\begin{tabular*}{\temptablewidth}{|c|c|cc|cc|cc|}
    $\alpha$ & $N$ & $\mathcal E_f(T)$  &  order  & $\mathcal E_{\widehat U} (T)$ &  order  & $\mathcal E_{W}(T)$ & order  \\
  \cline{1-8}
                 &  16  & 5.2364E-03  & -      &  9.7236E-06   &  -      & 8.4996E-03 & -           \\
           0.25  &  32  & 1.3418E-03  & 1.9644 &  2.5773E-06   &  1.9156 & 2.6395E-03 & 1.6871   \\
                 &  64  & 3.4081E-04  & 1.9771 &  6.6612E-07   &  1.9520 & 8.0758E-04 & 1.7086   \\
                 &  128 & 8.5987E-05  & 1.9868 &  1.6980E-07   &  1.9720 & 2.4454E-04 & 1.7235  \\
     \cline{1-8}
                 &  16  & 5.2587E-03  & -      &  9.8743E-06   &  -      & 7.3876E-03 & -       \\
           0.5   &  32  & 1.3389E-03  & 1.9737 &  2.5896E-06   &  1.9309 & 2.7286E-03 & 1.4369   \\
                 &  64  & 3.3884E-04  & 1.9823 &  6.6434E-07   &  1.9628 & 9.9896E-04 & 1.4497   \\
                 &  128 & 8.5404E-05  & 1.9882 &  1.6847E-07   &  1.9794 & 3.6376E-04 & 1.4614  \\
     \cline{1-8}
                 &  16  & 5.3287E-03  & -      &  1.0023E-05   &  -      & 5.5732E-03 & -        \\
           0.75  &  32  & 1.3698E-03  & 1.9598 &  2.6066E-06   &  1.9431 & 2.4161E-03 & 1.2058   \\
                 &  64  & 3.4926E-04  & 1.9716 &  6.6493E-07   &  1.9709 & 1.0391E-03 & 1.2173   \\
                 &  128 & 8.8536E-05  & 1.9800 &  1.6799E-07   &  1.9848 & 4.4334E-04 & 1.2289   \\
       \end{tabular*}
{\rule{\temptablewidth}{0.95pt}}
\end{center}
\end{table}

\section{Conclusion}

In this paper, we derive the optimal order a posteriori error estimates for the fully discrete scheme of TFPDEs, where the L1 algorithm and the finite element method are applied in time and space discretizations, respectively. In view of the weak regularity of the solutions to this class of equations, a posteriori error estimates are extremely important for solving TFPDEs. Firstly, with the help of  linear space-time reconstruction, we  derived a posteriori error estimate  in the  $L^\infty(L^2)$-norm and $L^1(L^2)$-norm for the fully discrete scheme. But the numerical experiments in Section \ref{sec:6} indicate that the residual of the  error equation for the linear space-time reconstruction  is of suboptimal order with respect to the L1 method even for the problem with smooth solution.  Then, to overcome the suboptimal convergence of residuals for the linear space-time reconstruction, we derive a posteriori error estimate for the quadratic reconstruction in the sense of $L^\infty(L^2)$-norm and $L^1(L^2)$-norm. Further, by means of  the above a posteriori error estimates and triangular inequality, global and pointwise final error estimates for the linear and quadratic space-time reconstruction are obtained. It is worth emphasizing that these a posteriori error estimates and final error estimates depend only on the numerical solution and the initial data of the problem, and do not depend on the true solution of problem, so they are also computable for problems with unavailable true solutions. Example 6.1 shows that a posteriori error estimators are of optimal  order  for the problem with smooth solution, but Example 6.2 indicates that using uniform mesh could only achieve  suboptimal order for the problem with nonsmooth data. Then we perform the  Example 6.2 again on the graded mesh and check that a posteriori error estimators recover the optimal convergence, and this can be seen from Table \ref{table1-5}-\ref{table1-6}. Therefore, for the linear and quadratic space-time reconstructions, these numerical results demonstrate that a posteriori error estimate of the fully discrete scheme is valid on nonuniform meshes regardless of whether the true solution of the problem is smooth or not.

For TFPDEs, it is the first time to give the complete theoretical analysis of a posteriori error estimates for the fully discrete scheme. In this paper, we only consider the a posteriori error estimates  but not adaptive algorithms. Designing suitable adaptive algorithms based on a posteriori error estimates and considering higher order numerical algorithms in the time direction will be our next work.

\section*{Appendixes}
\subsection*{Appendix 1} With the help of the elliptic reconstruction error,  the computable form of the term $\int_{0}^{T}\|\bar\partial^\alpha_k U^n_h-\partial^\alpha_t U_{\mathcal R}\|dt$ is investigated.

Applying  Proposition \ref{pro2.2} and Lemma \ref{lem4}, we can obtain from (\ref{eq2.15}) and (\ref{eq3.4})
\begin{equation*}
\begin{aligned}
\int_{0}^{T}\|\bar\partial^\alpha_k U^n_h-&\partial^\alpha_t U_{\mathcal R}\|dt=\sum\limits_{n=1}^N\int_{t_{n-1}}^{t_n}\|\bar\partial^\alpha_k U^n_h-\partial^\alpha_t U_{\mathcal R}\|dt\\
&\qquad=\sum\limits_{n=1}^N\int_{t_{n-1}}^{t_n}\| \sum\limits_{j=1}^{n}a_{j}(t_n)\Pi_j^n\bar{\partial}U_h^j  -\sum\limits_{j=1}^{n-1}a_j(t)\frac{\mathcal R^{j}U_h^j-\mathcal R^{j-1}U_h^{j-1}}{k_j}\\
&\qquad\quad- \frac{(t-t_{n-1})^{1-\alpha}}{\Gamma(2-\alpha)}\frac{\mathcal R^{n}U_h^n-\mathcal R^{n-1}U_h^{n-1}}{k_n} \|dt\\
&\qquad=\sum\limits_{n=1}^N\int_{t_{n-1}}^{t_n}\| \sum\limits_{j=1}^{n}a_{j}(t_n)\frac{(\Pi_j^n-I)U_h^j-(\Pi_{j-1}^n-I)U_h^{j-1}+(U_h^j-U_h^{j-1})}{k_j}\\
&\qquad\quad-\sum\limits_{j=1}^{n-1}a_j(t)\frac{(\mathcal R^{j}-I)U_h^j-(\mathcal R^{j-1}-I)U_h^{j-1}+(U_h^j-U_h^{j-1})}{k_j}\\
&\qquad\quad -\frac{(t-t_{n-1})^{1-\alpha}}{\Gamma(2-\alpha)}\frac{(\mathcal R^{n}-I)U_h^n-(\mathcal R^{n-1}-I)U_h^{n-1}+(U_h^n-U_h^{n-1})}{k_n} \|dt\\
&\qquad\leq \sum\limits_{n=1}^N\int_{t_{n-1}}^{t_n} \sum\limits_{j=1}^{n}\frac{a_{j}(t_n)}{k_j}( \mathcal E_{I0}(U_h^j)+\mathcal E_{I0}(U_h^{j-1}))+ \sum\limits_{j=1}^{n-1}\frac{a_j(t)}{k_j}( \mathcal E_{e0}(U_h^j)+\mathcal E_{e0}(U_h^{j-1}))\\
&\qquad\quad +\frac{(t-t_{n-1})^{1-\alpha}}{\Gamma(2-\alpha)k_n}( \mathcal E_{e0}(U_h^n)+\mathcal E_{e0}(U_h^{n-1}))  +\sum\limits_{j=1}^{n-1}\frac{a_{j}(t_n)-a_j(t)}{k_j}\|U_h^j-U_h^{j-1}\|\\
&\qquad\quad+ (\frac{a_{n}(t_n)}{k_n}- \frac{(t-t_{n-1})^{1-\alpha}}{\Gamma(2-\alpha)k_n})\|U_h^n-U_h^{n-1}\| \\
&\leq  \sum\limits_{n=1}^N \sum\limits_{j=1}^{n}\left( C_1^{j,n}( \mathcal E_{I0}(U_h^j)+\mathcal E_{I0}(U_h^{j-1}))+  C_2^{j,n}( \mathcal E_{e0}(U_h^n)+\mathcal E_{e0}(U_h^{n-1}))+  C_3^{j,n} \|U_h^j-U_h^{j-1}\|   \right),
\end{aligned}
\end{equation*}
where
\begin{equation*}
\begin{aligned}
&C_1^{j,n}=\int_{t_{n-1}}^{t_n} \frac{a_j(t_n)}{k_j}dt,~(j=1,\cdots,n;~n=1,\cdots,N;),\\
&C_2^{j,n}=\left\{
\begin{aligned}
\int_{t_{n-1}}^{t_n} \frac{a_j(t)}{k_j}dt,~(j=1,\cdots,n-1;~n=2,\cdots,N;),\\
\int_{t_{n-1}}^{t_n} \frac{(t-t_{n-1})^{1-\alpha}}{\Gamma(2-\alpha)k_n} dt,~(j=n;~n=1,\cdots,N;),
\end{aligned}
\right.\\
&C_3^{j,n}=C_1^{j,n}-C_2^{j,n},~(j=1,\cdots,n;~n=1,\cdots,N;).
\end{aligned}
\end{equation*}
Thus the term $\int_{0}^{T}\|\bar\partial^\alpha_k U^n_h-\partial^\alpha_t U_{\mathcal R}\|dt$ can be controlled according to the above equation. $\qquad\qquad\qquad\Box$

\subsection*{Appendix 2} In this part of Appendixes, we are committed to obtaining the computable expression of the term   $\int_0^T\|{\mathcal R}^nU_h^n-U_{\mathcal R}\|^2_1dt$.

By means of Lemma \ref{lem4} and the definition of $U_{\mathcal R}$, we could express $\|{\mathcal R}^nU_h^n-U_{\mathcal R}\|^2_1$ as
\begin{equation*}
\begin{aligned}
\|{\mathcal R}^nU_h^n-U_{\mathcal R}\|^2_1&=\|{\mathcal R}^nU_h^n-\ell_{n,-1}(t)\mathcal R^{n-1}U_h^{n-1}-\ell_{n,0}(t)\mathcal R^{n} U_h^n\|_1^2\\
&=\|(1-\ell_{n,0}(t))(\mathcal R^{n}-I) U_h^n+(1-\ell_{n,0}(t)) U_h^n-\ell_{n,-1}(t)(\mathcal R^{n-1}-I)U_h^{n-1}-\ell_{n,-1}(t)U_h^{n-1}\|_1^2\\
&=\|(1-\ell_{n,0}(t))(\mathcal R^{n}-I) (U_h^n-U_h^{n-1}) +(1-\ell_{n,0}(t))( U_h^n- U_h^{n-1})\|_1^2\\
& \leq 2(1-\ell_{n,0}(t))^2(t)\mathcal E_{e1}^2(U_h^{n}-U_h^{n-1})+2(1-\ell_{n,0}(t))^2\|U_h^{n}-U_h^{n-1}\|_1^2.
\end{aligned}
\end{equation*}
Thus we could compute the term $\int_0^T\|{\mathcal R}^nU_h^n-U_{\mathcal R}\|^2_1dt$ from the above inequality
\begin{equation*}
\begin{aligned}
\int_0^T\|{\mathcal R}^nU_h^n-U_{\mathcal R}\|^2_1dt  & \leq \sum\limits_{n=1}^{N}\int_{t_{n-1}}^{t_n}\|{\mathcal R}^nU_h^n-U_{\mathcal R}\|^2_1dt\\
&\leq \sum\limits_{n=1}^{N}\int_{t_{n-1}}^{t_n} 2(1-\ell_{n,0}(t))^2(t)\mathcal E_{e1}^2(U_h^{n}-U_h^{n-1})+2(1-\ell_{n,0}(t))^2\|U_h^{n}-U_h^{n-1}\|_1^2dt\\
&\leq \sum\limits_{n=1}^{N}\left\{C_4^n\mathcal E_{e1}^2(U_h^{n}-U_h^{n-1})+C_4^n\|U_h^{n}-U_h^{n-1}\|_1^2\right\},
\end{aligned}
\end{equation*}
where
\begin{eqnarray*}
 C_4^n=\int_{t_{n-1}}^{t^n} 2(1-\ell_{n,0}(t))^2 dt,~(n=1,\cdots,N).
\end{eqnarray*}
$\qquad\qquad\qquad\qquad\qquad\qquad \qquad\qquad\qquad\qquad\qquad\qquad\qquad\qquad\qquad\qquad \qquad\qquad\qquad\qquad\qquad\qquad\quad\Box$

\subsection*{Appendix 3}  Depending on the definition of $\widehat U_{\mathcal R}(t)$, we divide $\int_0^T\|U_{\mathcal R}(t)-\widehat U_{\mathcal R}(t)\|_1^2dt$ into $\int_0^{t_1}\|U_{\mathcal R}(t)-\widehat U_{\mathcal R}(t)\|_1^2dt$ ang $\int_{t_1}^T\|U_{\mathcal R}(t)-\widehat U_{\mathcal R}(t)\|_1^2dt$ for discussion. When  $t\in I_1$, by means of (\ref{eq5.1}), we have
\begin{equation*}
\begin{aligned}
\|U_{\mathcal R}(t)-\widehat U_{\mathcal R}(t)\|_1^2&=\frac{t^2}{k_1^2}a(\mathcal R^1 U_h^1-\mathcal R^0 U_h^0,\mathcal R^1 U_h^1-\mathcal R^0 U_h^0)\\
&-\frac{2t^{\alpha+2}}{\Gamma(\alpha+2)k_1^2}a(\mathcal R^1 U_h^1-\mathcal R^0 U_h^0,f_h^1-f_h^0-\mathcal A ( \mathcal R^1U_h^1  -\mathcal R^0U_h^0))\\
&-\frac{2t^{\alpha+1}}{\Gamma(\alpha+1)k_1}a(\mathcal R^1 U_h^1-\mathcal R^0 U_h^0,f_h^0- \mathcal A \mathcal R^0U_h^0)+\frac{t^{2\alpha}}{{\Gamma^2(\alpha+1)}}a(f_h^0- \mathcal A \mathcal R^0U_h^0,f_h^0- \mathcal A \mathcal R^0U_h^0)\\
&+\frac{2t^{2\alpha+1}}{\Gamma(\alpha+1)\Gamma(\alpha+2)k_1}a(f_h^0- \mathcal A \mathcal R^0U_h^0,f_h^1-f_h^0-\mathcal A ( \mathcal R^1U_h^1  -\mathcal R^0U_h^0))\\
&+\frac{t^{2\alpha+2}}{\Gamma^2(\alpha+2)k_1^2}a(f_h^1-f_h^0-\mathcal A ( \mathcal R^1U_h^1  -\mathcal R^0U_h^0),f_h^1-f_h^0-\mathcal A ( \mathcal R^1U_h^1  -\mathcal R^0U_h^0)).
\end{aligned}
\end{equation*}
Further from the above inequality, we get
\begin{equation*}
\begin{aligned}
\int_0^{t_1}\|U_{\mathcal R}(t)-\widehat U_{\mathcal R}(t)\|_1^2dt&=\frac{k_1}{3}a(\mathcal R^1 U_h^1-\mathcal R^0 U_h^0,\mathcal R^1 U_h^1-\mathcal R^0 U_h^0)\\
&-\frac{2k_1^{\alpha+1}}{(\alpha+2)\Gamma(\alpha+1)}a(\mathcal R^1 U_h^1-\mathcal R^0 U_h^0,f_h^0- \mathcal A \mathcal R^0U_h^0)\\
&-\frac{2k_1^{\alpha+1}}{(\alpha+3)\Gamma(\alpha+2)}a(\mathcal R^1 U_h^1-\mathcal R^0 U_h^0,f_h^1-f_h^0-\mathcal A ( \mathcal R^1U_h^1  -\mathcal R^0U_h^0))\\
&+\frac{k_1^{2\alpha+1}}{(2\alpha+1)\Gamma^2(\alpha+1)}a(f_h^0- \mathcal A \mathcal R^0U_h^0,f_h^0- \mathcal A \mathcal R^0U_h^0)\\
&+\frac{k_1^{2\alpha+1}}{\Gamma^2(\alpha+2)}a(f_h^0- \mathcal A \mathcal R^0U_h^0,f_h^1-f_h^0-\mathcal A ( \mathcal R^1U_h^1  -\mathcal R^0U_h^0))\\
&+\frac{k_1^{2\alpha+1}}{(2\alpha+3)\Gamma^2(\alpha+2)}a(f_h^1-f_h^0-\mathcal A ( \mathcal R^1U_h^1  -\mathcal R^0U_h^0),f_h^1-f_h^0-\mathcal A ( \mathcal R^1U_h^1  -\mathcal R^0U_h^0)).
\end{aligned}
\end{equation*}

 When $t\in I_n~(n\geq 2)$, applying Lemma \ref{lem4} and Proposition \ref{pro2.2} to (\ref{eq4.1}), we have
\begin{equation*}
\begin{aligned}
\|U_{\mathcal R}(t)-\widehat U_{\mathcal R}(t)\|_1^2&=\|\frac{(t-t_{n-1})(t-t_n)}{k_{n}+k_{n-1}}\mathcal R^n(\bar\partial U^{n}_h-\Pi^n_{n-1}\bar\partial U^{n-1}_h)\|_1^2\\
&\leq \|\frac{(t-t_{n-1})(t-t_n)}{k_{n}+k_{n-1}}\left\{  (\mathcal R^n-I)\bar\partial U^{n}_h+\bar\partial U^{n}_h- (\mathcal R^n-I)\Pi^n_{n-1}\bar\partial U^{n-1}_h)-\Pi^n_{n-1}\bar\partial U^{n-1}_h)\right\}\|_1^2\\
&\leq \left(\frac{(t-t_{n-1})(t_n-t)}{k_{n}+k_{n-1}}\right)^2\left\{4\mathcal E_{e1}(\bar\partial U^{n}_h)
+4C_{\Pi}\mathcal E_{e1}(\bar\partial U^{n-1}_h)+4\|\bar\partial U^{n}_h\|_1^2+4C_{\Pi}\|\bar\partial U^{n-1}_h\|_1^2\right\}.
\end{aligned}
\end{equation*}
Based on the above inequality, it is easy to obtain
\begin{equation*}
\begin{aligned}
\int_{t_1}^T\|U_{\mathcal R}(t)-\widehat U_{\mathcal R}(t)\|_1^2dt
&\leq \sum\limits_{n=2}^N\int_{t_{n-1}}^{t_n}\left(\frac{(t-t_{n-1})(t_n-t)}{k_{n}+k_{n-1}}\right)^2dt \left\{4\mathcal E_{t,1}(\bar\partial U^{n}_h)\right.\\
&\quad\left.+4C_{\Pi}\mathcal E_{t,1}(\bar\partial U^{n-1}_h)+4\|\bar\partial U^{n}_h\|_1^2+4C_{\Pi}\|\bar\partial U^{n-1}_h\|_1^2\right\}\\
&\leq \frac{k_n^5}{30(k_{n}+k_{n-1})^2}  \left\{4\mathcal E_{e1}(\bar\partial U^{n}_h)
+4C_{\Pi}\mathcal E_{e1}(\bar\partial U^{n-1}_h)+4\|\bar\partial U^{n}_h\|_1^2+4C_{\Pi}\|\bar\partial U^{n-1}_h\|_1^2\right\}.
\end{aligned}
\end{equation*}
Therefore, we give the computable expression of the term $\int_{0}^T\|U_{\mathcal R}(t)-\widehat U_{\mathcal R}(t)\|_1^2dt$.$\qquad\qquad\qquad\qquad\qquad \Box $

\subsection*{Appendix 4} By analyzing the elliptic reconstruction error,  a computable expression for the term $\int_{t_1}^T\|E_{\mathcal I}^n(t)\|dt$ can be obtained.  Applying Lemma \ref{lem4} to the definition of $E_{\mathcal I}^n(t)$  yields
\begin{equation*}
\begin{aligned}
\int_{t_1}^T\|E_{\mathcal I}^n(t)\|dt&=\sum\limits_{j=2}^{N}\int_{t_{n-1}}^{t_n}\|\mathcal A U_{\mathcal R}-A_h^nU_h^n\|dt\\
&=\sum\limits_{j=2}^{N}\int_{t_{n-1}}^{t_n}\|\ell_{n,-1}(t)\mathcal A \mathcal R^{n-1}U_h^{n-1}+\ell_{n,0}(t)\mathcal A \mathcal R^{n}U_h^{n}-\mathcal A_h^nU_h^n\|dt\\
&=\sum\limits_{j=2}^{N}\int_{t_{n-1}}^{t_n}\|\ell_{n,-1}(t) (\mathcal R^{n-1}\mathcal AU_h^{n-1}-\mathcal AU_h^{n-1}+\mathcal AU_h^{n-1})+\ell_{n,0}(t)(\mathcal R^{n}\mathcal AU_h^{n}-\mathcal AU_h^{n}+\mathcal AU_h^{n})\\
&\quad-(\ell_{n,-1}(t)+\ell_{n,0}(t))\mathcal A_h^nU_h^n\|dt\\
&\leq \sum\limits_{j=2}^{N}\int_{t_{n-1}}^{t_n} \ell_{n,-1}(t) (\mathcal E_{e0}(\mathcal AU_h^{n-1})+(\mathcal A-\mathcal A_h^n)U_h^{n-1}+\mathcal A_h^n(U_h^{n-1}-U_h^n))dt\\
&\quad+\sum\limits_{j=2}^{N}\int_{t_{n-1}}^{t_n}\ell_{n,0}(t) (\mathcal E_{e0}(\mathcal AU_h^{n}+(\mathcal A-\mathcal A_h^n)U_h^n))dt\\
&\leq \sum\limits_{j=2}^{N} \frac{k_n}{2}\left( \mathcal E_{e0}(\mathcal AU_h^{n-1})+\mathcal E_{e0}(\mathcal AU_h^{n})+(\mathcal A-\mathcal A_h^n)(U_h^{n-1}+U_h^n)+\mathcal A_h^n(U_h^{n-1}-U_h^n)\right).
\end{aligned}
\end{equation*}
Therefore, we can compute the term $\int_{t_1}^T\|E_{\mathcal I}^n(t)\|dt$ according to the above inequality. $\qquad\qquad\qquad\qquad \Box$

\subsection*{Appendix 5} It's the turn to discuss a computable expression of the term $\int_{t_1}^T\|E_{\mathcal W}^n(t)\|dt$. From the definition of  $\|E_{\mathcal W}^n(t)\|$ in (\ref{eq4.8}), it is clear that we need to estimate term $\|{\mathcal R}^n\widehat W_h^n\|$ to get its computable expression. Then the term $\|{\mathcal R}^n\widehat W_h^n\|$ can be processed as follows
\begin{equation*}
\begin{aligned}
&\|{\mathcal R}^n\widehat W_h^n\| = {\mathcal R}^n\frac{2(\bar\partial U^{n}_h-\Pi^n_{n-1}\bar\partial U^{n-1}_h)}{k_{n}+k_{n-1}}\|\\
&\quad= \frac{2}{k_{n}+k_{n-1}}\left\{ ({\mathcal R}^n - I+I)\| \bar\partial U^{n}_h- \bar\partial U^{n-1}_h \|+ ({\mathcal R}^n - I+I)\|(\Pi^n_{n-1}-I)\bar\partial U^{n-1}_h)\|  \right\}\\
&\quad\leq \frac{2}{k_{n}+k_{n-1}}\left\{ \mathcal E_{e0}(\bar\partial U^{n}_h- \bar\partial U^{n-1}_h)+\| \bar\partial U^{n}_h - \bar\partial U^{n-1}_h\|+ \mathcal E_{e0}((\Pi^n_{n-1}-I)\bar\partial U^{n-1}_h)+\|(\Pi^n_{n-1}-I) \bar\partial U^{n-1}_h \|  \right\}\\
&\quad\leq \frac{2}{k_{n}+k_{n-1}}\left\{ \mathcal E_{e0}(\bar\partial U^{n}_h- \bar\partial U^{n-1}_h)+\| \bar\partial U^{n}_h- \bar\partial U^{n-1}_h \|+ C_{\Pi}\mathcal E_{e0}(\bar\partial U^{n-1}_h)+C_{\Pi}\|\bar\partial U^{n-1}_h \|  \right\}.
\end{aligned}
\end{equation*}
In conclusion, we give the feasible method to compute the term $E_{\mathcal W}^n(t)$. In fact,
\begin{equation*}
\begin{aligned}
\|\mathcal W_h^n\|&\leq \frac{1}{2\Gamma(2-\alpha)}\sum\limits_{j=2}^{n-1}
k_j\left[(t-t_j)^{1-\alpha}+(t-t_{j-1})^{1-\alpha}\right]\|{\mathcal R}^j\widehat W_h^j\|\\
&+\frac{1}{\Gamma(3-\alpha)}\sum\limits_{j=2}^{n-1}\widehat W_h^j\left[(t-t_j)^{2-\alpha}-(t-t_{j-1})^{2-\alpha}\right]\|{\mathcal R}^j\widehat W_h^j\|\\
&+\frac{k_n(t-t_{n-1})^{1-\alpha}}{2\Gamma(2-\alpha)}\|{\mathcal R}^n\widehat W_h^n\|+\frac{(t-t_{n-1})^{2-\alpha}}{\Gamma(3-\alpha)}\|{\mathcal R}^n\widehat W_h^n\|\\
&+\frac{1}{\Gamma(1-\alpha)}\int^{t_{1}}_{0}(t-s)^{-\alpha}\left(\frac{s^\alpha}{\Gamma(\alpha+1)k_1}\|{f_h^1-f_h^0}-{\mathcal A ( \mathcal R^1U_h^1  -\mathcal R^0U_h^0)}\|+\frac{s^{\alpha-1}}{\Gamma(\alpha)}\|f_h^0-{\mathcal R}^0 \mathcal A U_h^0\|\right)ds\\
&+\frac{t^{1-\alpha}-(t-t_1)^{1-\alpha}}{\Gamma(2-\alpha)k_1}  \|\mathcal R^1 U_h^1-\mathcal R^0 U_h^0\|.
\end{aligned}
\end{equation*}
Thus we could compute the term $\int_{t_1}^T\|E_{\mathcal W}^n(t)\|dt$ as follow
\begin{equation*}
\begin{aligned}
&\int_{t_1}^T\|\mathcal W_h^n\|dt\leq \sum\limits_{n=3}^{N}\sum\limits_{j=2}^{n-1}\int_{t_{n-1}}^{t_n}\frac{k_j}{2\Gamma(2-\alpha)}
\left((t-t_j)^{1-\alpha}+(t-t_{j-1})^{1-\alpha}\right)dt\|{\mathcal R}^j\widehat W_h^j\|\\
&+\sum\limits_{n=3}^{N}\sum\limits_{j=2}^{n-1}\int_{t_{n-1}}^{t_n}\frac{1}{\Gamma(3-\alpha)}\left[(t-t_j)^{2-\alpha}-(t-t_{j-1})^{2-\alpha}\right]dt\|{\mathcal R}^j\widehat W_h^j\|\\
&+\sum\limits_{n=2}^{N}\int_{t_{n-1}}^{t_n}(\frac{k_n(t-t_{n-1})^{1-\alpha}}{2\Gamma(2-\alpha)}+\frac{(t-t_{n-1})^{2-\alpha}}{\Gamma(3-\alpha)})dt\|{\mathcal R}^n\widehat W_h^n\|+\int_{t_1}^T\frac{t^{1-\alpha}-(t-t_1)^{1-\alpha}}{\Gamma(2-\alpha)k_1}  \|\mathcal R^1 U_h^1-\mathcal R^0 U_h^0\|dt\\
&+\frac{1}{\Gamma(1-\alpha)}\int_{t_1}^T\int^{t_{1}}_{0}(t-s)^{-\alpha}\left(\frac{s^\alpha}{\Gamma(\alpha+1)k_1}\|{f_h^1-f_h^0}-{\mathcal A ( \mathcal R^1U_h^1  -\mathcal R^0U_h^0)}\|+\frac{s^{\alpha-1}}{\Gamma(\alpha)}\|f_h^0-{\mathcal R}^0 \mathcal A U_h^0\|\right)dsdt\\
&\leq \sum\limits_{n=3}^{N}\sum\limits_{j=2}^{n-1}C_7^{j,n}\|{\mathcal R}^j\widehat W_h^j\| + \sum\limits_{n=2}^{N}C_8^n\|{\mathcal R}^n\widehat W_h^n\|+C_9\|{f_h^1-f_h^0}-{\mathcal A ( \mathcal R^1U_h^1  -\mathcal R^0U_h^0)}\|+C_{10}\|f_h^0-{\mathcal R}^0 \mathcal A U_h^0\|\\
&\quad+C_{11}\|\mathcal R^1 U_h^1-\mathcal R^0 U_h^0\|,
\end{aligned}
\end{equation*}
where
\begin{equation*}
\begin{aligned}
&C_7^{j,n}=\int_{t_{n-1}}^{t_n} \frac{k_j}{2\Gamma(2-\alpha)}
\left((t-t_j)^{1-\alpha}+(t-t_{j-1})^{1-\alpha}\right)dt\\
&\quad\quad\quad+ \int_{t_{n-1}}^{t_n}\frac{1}{\Gamma(3-\alpha)}\left((t-t_j)^{2-\alpha}-(t-t_{j-1})^{2-\alpha}\right)dt, ~(j=2,\cdots,n-1;~n=3,\cdots,N;)\\
&C_8^n=\int_{t_{n-1}}^{t_n}  (\frac{k_n(t-t_{n-1})^{1-\alpha}}{2\Gamma(2-\alpha)}+\frac{(t-t_{n-1})^{2-\alpha}}{\Gamma(3-\alpha)})   dt,~(n=2,\cdots,N;)\\
&C_9= \frac{1}{\Gamma(1-\alpha)\Gamma(\alpha+1)k_1}\int_{t_1}^T \int^{t_{1}}_{0}  (t-s)^{-\alpha} s^\alpha   dsdt,\\
&C_{10}= \frac{1}{\Gamma(1-\alpha)\Gamma(\alpha)}\int_{t_1}^T \int^{t_{1}}_{0}  (t-s)^{-\alpha} s^{\alpha -1}  dsdt,\\
&C_{11}=\int_{t_1}^T  \frac{t^{1-\alpha}-(t-t_1)^{1-\alpha}}{\Gamma(2-\alpha)k_1} dt.
\end{aligned}
\end{equation*}
$\qquad\qquad\qquad\qquad\qquad\qquad \qquad\qquad\qquad\qquad\qquad\qquad\qquad\qquad\qquad\qquad \qquad\qquad\qquad\qquad\qquad\qquad\quad\Box$

\bibliographystyle{amsplain}

\enlargethispage{1\baselineskip}
\end{document}